\documentclass[notitlepage,11pt,reqno]{amsart}
\usepackage{amssymb,amsmath,tocvsec2,color,mathrsfs,enumerate}
\usepackage[breaklinks=true]{hyperref}
\usepackage[letterpaper]{geometry}
\theoremstyle{plain}
\newtheorem{theorem}[equation]{Theorem}
\newtheorem{lemma}[equation]{Lemma}
\newtheorem{prop}[equation]{Proposition}
\newtheorem{cor}[equation]{Corollary}

\newtheorem{utheorem}{\textrm{\textbf{Theorem}}}

\theoremstyle{definition}
\newtheorem{defn}[equation]{Definition}
\newtheorem{remark}[equation]{Remark}
\newtheorem{qn}[equation]{Question}

\numberwithin{equation}{section}

\newcommand{\wt}{\operatorname{wt}}
\newcommand{\conv}{\operatorname{conv}}
\newcommand{\Fin}{\operatorname{Fin}}
\newcommand{\supp}{\operatorname{supp}}
\newcommand{\hgt}{\operatorname{ht}}
\newcommand{\ch}{\operatorname{ch}}

\newcommand{\disp}{\overset{\rightarrow}{\ell}}
\newcommand{\V}{\ensuremath{\mathbb{V}}}
\newcommand{\vla}{\ensuremath{\mathbb{V}^\lambda}}
\newcommand{\F}{\ensuremath{\mathbb{F}}}
\newcommand{\R}{\ensuremath{\mathbb{R}}}
\newcommand{\C}{\ensuremath{\mathbb{C}}}
\newcommand{\bba}{\ensuremath{\mathbb{A}}}
\newcommand{\Z}{\ensuremath{\mathbb{Z}}}
\newcommand{\calp}{\ensuremath{\mathcal{P}}}
\newcommand{\liehr}{\ensuremath{\mathfrak{h}_{\mathbb{R}}}}
\newcommand{\wtvla}[1]{\ensuremath{\wt_{#1} L(\lambda)}}
\newcommand{\lie}[1]{\ensuremath{\mathfrak{#1}}}

\begin{document}
\title{Faces and maximizer subsets of highest weight modules}

\author{Apoorva Khare}

\dedicatory{Dedicated to the memory of D.-N.~Verma,\\
from whom I first learned representation theory}

\email[A.~Khare]{\tt khare@stanford.edu}

\address{Departments of Mathematics and Statistics, Stanford University,
Stanford, CA 94305, USA}

\date{\today}

\subjclass[2010]{Primary: 17B10; Secondary: 17B20, 52B15, 52B20}

\keywords{Weak $\bba$-face, (pure) highest weight module, Weyl polytope,
parabolic Verma module}

\begin{abstract}
In this paper we study general highest weight modules
$\mathbb{V}^\lambda$ over a complex finite-dimensional semisimple Lie
algebra $\mathfrak{g}$. We present three formulas for the set of weights
of a large family of modules $\mathbb{V}^\lambda$, which include but are
not restricted to all simple modules and all parabolic Verma modules.
These formulas are direct and do not involve cancellations, and were not
previously known in the literature. Our results extend the notion of the
Weyl polytope to general highest weight $\mathfrak{g}$-modules
$\mathbb{V}^\lambda$.

We also show that for all simple modules, the convex hull of the weights
is a $W_J$-invariant polyhedron for some parabolic subgroup $W_J$. We
compute its vertices, faces, and symmetries -- more generally, we also do
this for all parabolic Verma modules, and for all modules
$\mathbb{V}^\lambda$ with highest weight $\lambda$ not on a simple root
hyperplane. To show our results, we extend the notion of convexity to
arbitrary additive subgroups $\mathbb{A} \subset (\mathbb{R},+)$ of
coefficients. Our techniques enable us to completely classify ``weak
$\mathbb{A}$-faces'' of the support sets ${\rm wt}(\mathbb{V}^\lambda)$,
in the process extending classical results of Satake, Borel--Tits,
Vinberg, and Casselman, as well as modern variants by
Chari--Dolbin--Ridenour and Cellini--Marietti, to general highest weight
modules.
\end{abstract}
\maketitle

\settocdepth{section}
{\tableofcontents}

\section{Introduction}

This paper contributes to the study of highest weight modules over a
complex finite-dimensional semisimple Lie algebra. Some of these, such as
finite-dimensional simple modules and (generalized/parabolic) Verma
modules, are classical and well understood. However, more work needs to
be done for infinite-dimensional ``non-Verma'' highest weight modules
(and even for finite-dimensional modules). Important questions such as
the set of weights of these modules, or the multiplicities of these
weights are not fully resolved as yet.

Fix a complex finite-dimensional semisimple Lie algebra $\lie{g}$, a set
of simple roots $\Delta$ in the space $\lie{h}^*$ of weights, the
associated Weyl group $W$ and root space decomposition for $\lie{g}$, and
an arbitrary weight $\lambda \in \lie{h}^*$.
In this paper we present three formulas for computing the supports --
i.e., sets of weights -- of all simple highest weight modules, denoted
hereafter by $L(\lambda)$.
While the support of simple modules $L(\lambda)$ was known for special
sub-families (dominant integral $\lambda$, or antidominant $\lambda$),
the answer was not known in general.
Our formulas are direct and do not involve any cancellations, unlike the
more sophisticated Weyl Character Formula for computing weight
multiplicities. Moreover, one of our formulas uses finite-dimensional
submodules for a distinguished Levi subalgebra, while another is in terms
of the convex hull of the weights (i.e., the set of all convex linear
combinations of weights). See Corollary \ref{Csimple}, and more
generally, Theorem \ref{Twtgvm} in Section \ref{S2}.

In this paper we are interested in studying the larger family of general
\textit{highest weight modules}. These are precisely the quotients of
Verma modules $M(\lambda)$, and we will denote such a module by $\vla$ to
indicate its highest weight. Thus, $M(\lambda) \twoheadrightarrow \vla
\twoheadrightarrow L(\lambda)$.
An obvious consequence of our formulas for the support of simple modules
$L(\lambda)$ is that the support of a general module $\vla$ is determined
from the multiplicities $[\vla : L(w \bullet \lambda)]$ of its
Jordan--Holder factors (which lie in the Bernstein--Gelfand--Gelfand
Category $\mathcal{O}$).
A more direct attempt to compute $\wt \vla$ is to prove an analogue of
Corollary \ref{Csimple} or Theorem \ref{Twtgvm} for all $\vla$. However,
these results fail to hold for all $\vla$; see Theorem \ref{Twtsimple}
below. Nevertheless, the techniques used in proving Theorem \ref{Twtgvm}
yield many other rewards, including:
\begin{itemize}
\item Computing the weights and their convex hulls, for other families of
highest weight modules $\vla$. These modules $\vla$ are
infinite-dimensional, whence their sets of weights $\wt \vla$ are
infinite. We are nevertheless able to show that their convex hulls are
polyhedra -- i.e., \textit{finite} intersections of half-spaces in
Euclidean space. This includes all Verma and simple modules.

\item Classifying the symmetries and faces of these convex hulls, and (in
related work \cite{Khminmax},) classifying all inclusion relations
between these faces.

\item Classical results in the literature by Satake, Borel--Tits,
Vinberg, and Casselman, which were known only for finite-dimensional
simple modules, are now shown for all highest weight modules.

\item A longer-term goal involves computing weight multiplicities of
highest weight modules. We are able to obtain some results along these
lines, by extending the Weyl character formula under somewhat different
assumptions than in the literature. See Theorem \ref{Twcf} and the
preceding remarks.
\end{itemize}

Another feature of this paper is to focus on several important families
of highest weight modules that feature prominently in the literature:
\begin{enumerate}[(i)]
\item Parabolic Verma modules, which include all Verma modules.
\item All simple highest weight modules $L(\lambda)$.
\item All highest weight modules $\vla$ with $\lambda$ not on a simple
root hyperplane. These include all antidominant weights $\lambda$ (whence
$\vla = M(\lambda) = L(\lambda)$) as well as all regular weights
$\lambda$.
\end{enumerate}

\noindent We also consider a fourth class of highest weight modules
termed ``pure'' modules. These modules feature in the classification of
all simple $\lie{h}$-weight $\lie{g}$-modules, in work of Fernando
\cite{Fe}. In this paper we provide a wide variety of techniques for
studying all of these families of modules. Thus, a module that lies in
more than one of these families can be studied in more than one way. For
instance, two of our main results, Theorems \ref{T2} and \ref{T3}, hold
for four different kinds of highest weight modules. Corresponding to
these, there are multiple proof techniques presented in this paper.

\subsection{Organization of the paper}

We now briefly outline the rest of the paper. In Section \ref{S2}, we set
notation and write down the main results of the paper, discussing along
the way several motivating questions and related results in the
literature.
The remaining sections are dedicated to proving these main results.
In particular, as discussed above, we present several approaches to
proving the main results for different families of highest weight
modules. For instance, the algebraic approach in Section \ref{S4} avoids
using the Weyl group in studying modules $\vla$ when $\lambda$ avoids the
simple root hyperplanes; this approach provides an alternate proof of
some of the results in \cite{KR} for all highest weight modules $\vla$
over a dense set of weights $\lambda$.
In Sections \ref{S5} and \ref{Sfer}, we prove three of the main results
on the structure of highest weight modules $\vla$, by computing the
convex hull, stabilizer subgroup, and vertices of (the hull of) the
weights of $\vla$.
There are also two applications of our techniques and results. In Section
\ref{Sappl1}, we compute the support of all simple modules $L(\lambda)$
(and others). In Section \ref{Sappl2}, we compute the unique ``largest''
and ``smallest'' highest weight modules with specified convex hull of
weights.

\section{Main results and literature survey}\label{S2}

We now describe the main results of the paper on general highest weight
modules $\vla$. Along the way, we discuss several motivating questions
and related results in the literature, which, somewhat surprisingly, only
study simple finite-dimensional modules and parabolic Verma modules.

\subsection{Notation and preliminaries}

We begin by writing down some basic notation and results on linear
combinations and on Verma modules; these will be freely used without
reference in what follows.
Let $\R \supset \F \supset \mathbb{Q} \supset \Z$ denote the real
numbers, a (possibly fixed) subfield, the rationals, and the integers
respectively. Given an $\R$-vector space $\V$ and $R \subset \R,\ X,Y
\subset \V$, define $X \pm Y$ to be their Minkowski sum $\{ x \pm y : x
\in X, y \in Y \}$, $R_+ := R \cap [0,\infty)$, and $RX$ to be the set of
all finite linear combinations $\sum_{i=1}^k r_i x_i$, where $r_i \in R$
and $x_i \in X$. (This includes the empty sum $0$ if $k=0$.) Let
$\conv_\R(X)$ denote the set of convex $\R_+$-linear combinations of $X$.

Let $\lie g$ be a complex finite-dimensional semisimple Lie algebra with
a fixed triangular decomposition $\lie g = \lie n^+ \oplus \lie h \oplus
\lie n^-$. Let the corresponding root system be $\Phi$, with simple roots
$\Delta := \{ \alpha_i : i \in I \}$ and corresponding fundamental
weights $\Omega := \{ \omega_i : i \in I \}$ both indexed by $I$.
For any $J \subset I$, define $\Delta_J := \{ \alpha_j : j \in J \}$, and
$\Omega_J$ similarly. Set $\rho_J := \sum_{j \in J} \omega_j$, and define
$W_J$ to be the subgroup of the Weyl group $W$ (of $\lie g$), generated
by the simple reflections $\{ s_j = s_{\alpha_j} : j \in J \}$.
Let $\liehr^*$ be the real form of $\lie{h}^*$ -- i.e., the $\R$-span of
$\Delta$. Then $\liehr^* = \R \Omega$ as well. The \textit{height} of a
weight $\mu = \sum_{i \in I} r_i \alpha_i \in \liehr^*$ is defined as
$\hgt \mu := \sum_i r_i$. Moreover, $\lie{h}^*$ has a standard partial
order via: $\lambda \geq \mu$ if $\lambda - \mu \in \Z_+ \Delta$. Now let
$P := \Z \Omega \supset Q := \Z \Delta$ be the weight and root lattices
in $\liehr^*$ respectively, and define
\begin{equation}\label{Edef}
P^+_J := \Z_+ \Omega_J, \quad Q^+_J := \Z_+ \Delta_J, \quad P^+ := P^+_I,
\quad Q^+ := Q^+_I, \quad \Phi^\pm_J := \Phi \cap \pm Q^+_J, \quad
\Phi^\pm := \Phi^\pm_I.
\end{equation}

\noindent Thus, $P^+ = P^+_I$ is the set of dominant integral weights.
Let $(,)$ be the positive definite symmetric bilinear form on $\liehr^*$
induced by the restriction of the Killing form on $\lie g$ to $\liehr$.
Then $(\omega_i, \alpha_j) = \delta_{i,j} (\alpha_j, \alpha_j) / 2\
\forall i,j \in I$.
Define $h_i$ to be the unique element of $\lie{h}$ identified with $(2 /
(\alpha_i, \alpha_i)) \alpha_i$ via the Killing form. The $h_i$ form
a basis of $\lie{h}_\R$. Now fix a set of Chevalley generators $\{
x_{\alpha_i}^\pm \in \lie n^\pm : i \in I \}$ such that $[x_{\alpha_i}^+,
x_{\alpha_j}^-] = \delta_{ij} h_i$ for all $i,j \in I$. Also extend $(,)$
to all of $\lie{h}^*$. Then,
\begin{equation}\label{Efacts}
\alpha_i(h_i) = 2, \quad \omega_j(h_i) = \delta_{i,j}, \quad \lambda(h_i)
= \frac{2 (\lambda, \alpha_i)}{(\alpha_i, \alpha_i)}, \qquad \forall i,j
\in I,\ \lambda \in \lie{h}^*.
\end{equation}

Define $M(\lambda)$ to be the Verma module with highest weight $\lambda
\in \lie{h}^*$. In other words, set $M(\lambda) := U\lie{g} / U
\lie{g}(\lie{n}^+ + \ker \lambda)$.
This is an $\lie{h}$-semisimple, cyclic $\lie{g}$-module which has a
unique simple quotient $L(\lambda)$. Moreover, $M(\lambda)$ is
``universal'' among the set of $\lie{g}$-modules generated by a vector of
weight $\lambda$ that is killed by $\lie{n}^+$. Every module in this
latter set is called a \textit{highest weight module} and we will denote
a typical such module by $\vla$. Thus, $M(\lambda) \twoheadrightarrow
\vla \twoheadrightarrow L(\lambda)$. Additionally, $M(\lambda)$ has a
finite Jordan--Holder series. The composition factors are necessarily of
the form $L(w \bullet \lambda)$ with $\lambda - w \lambda \in \Z_+
\Delta$, where $\bullet$ denotes the twisted action of the Weyl group on
$\lie{h}^*$: $w \bullet \lambda := w(\lambda + \rho_I) - \rho_I$.

Finally, the $\lambda$-weight space of an $\lie{h}$-module $M$ is
$M_\lambda := \{ m \in M : h m = \lambda(h) m\ \forall h \in \lie{h} \}$.
We say that $M$ is a ($\lie{h}$-)weight module if $M = \bigoplus_{\lambda
\in \lie{h}^*} M_\lambda$. If moreover $\dim M_\lambda < \infty\ \forall
\lambda \in \lie{h}^*$, the formal character of $M$ is defined to be $\ch
M := \sum_{\lambda \in \lie{h}^*} (\dim M_\lambda) e^\lambda \in
\Z_+^{\lie{h}^*}$.
Submodules and quotient modules of weight modules are weight modules. It
is clear that $M(\lambda)$ is a weight module with finite-dimensional
weight spaces. Moreover, $M(\lambda)$ is a free $U(\lie{n}^-)$-module of
rank one by the PBW theorem, whose weights are precisely $\lambda - Q^+
= \lambda - \Z_+ \Delta$ and whose formal character is given by the
(translated) Kostant partition function. For a thorough treatment of
Verma modules and their simple quotients (as well as a distinguished
category $\mathcal{O}$ in which they all lie), the reader is referred to
the comprehensive book by Humphreys \cite{H3}.

\subsection{Motivation 1: weights and their hulls of simple and Verma
modules}

Our first motivation comes from the classical question of computing the
support and weight multiplicities of highest weight modules $\vla$. It
turns out that not much is known about simple modules $L(\lambda)$ (or
highest weight modules other than parabolic Verma modules), save for two
special families of simple modules. The first is the set of
\textit{antidominant} highest weights $\lambda$ -- i.e., $2(\lambda +
\rho_I, \alpha)/(\alpha,\alpha) - 1 \notin \Z_+$ for all $\alpha \in
\Phi^+$. In this case, $M(\lambda)$ is simple, e.g. by \cite[Theorem
4.8]{H3}. It follows that $\wt L(\lambda) = \wt M(\lambda) = \lambda -
\Z_+ \Delta$, and one verifies this equals $(\lambda - \Z \Delta) \cap
\conv_\R \wt L(\lambda)$.

The interesting phenomena occur at the ``opposite end'' (this is made
precise presently), for dominant integral $\lambda$, i.e., $\lambda \in
P^+$. Simple modules for such $\lambda$ yield symmetries, combinatorial
formulas, as well as crystals. It is standard (see \cite[Chapter 2]{H3})
that $\dim L(\lambda) < \infty$ if and only if $\lambda \in P^+$, in
which case,
\[
L(\lambda) = M(\lambda) \ / \ \sum_{i \in I} U \lie{g}
(x^-_{\alpha_i})^{\lambda(h_i) + 1} m_\lambda.
\]

We now state two results that will be used repeatedly in the paper. The
first of these results provides a recipe to compute the support of any
finite-dimensional highest weight module $L(\lambda)$, via its convex
hull.

\begin{theorem}\label{Tklv}
Notation as above. Define $\calp(\lambda) := \conv_\R \wtvla{}$ for
$\lambda \in \lie{h}^*$. Now fix $\lambda,\mu \in P^+$.
\begin{enumerate}
\item (\cite[Proposition 7.13 and Theorem 7.41]{Ha}.)
$\calp(\lambda)$ equals the convex hull of the Weyl orbit of $\lambda$:
$\calp(\lambda) = \conv_\R W(\lambda)$, where $W(\lambda) := \{
w(\lambda) : w \in W \}$ is the Weyl group orbit of $\lambda$. Moreover,
$\wtvla{} = (\lambda - \Z \Delta) \cap \calp(\lambda)$.

\item (\cite[Proposition 2.2]{KLV}.)
$\lambda - \mu \in \Z_+ \Delta$ if and only if $\conv_\R W(\mu) \subset
\conv_\R W(\lambda)$.
\end{enumerate}
\end{theorem}

\noindent Notice by the above discussion that $\wtvla{} = (\lambda - \Z
\Delta) \cap \calp(\lambda)$ for antidominant $\lambda$ as well, since in
this case $L(\lambda)$ is the Verma module. Given these two families of
simple modules, the following question is natural (and was posed to us by
D.~Bump):

\begin{qn}\label{Qbump}
Is it true that $\wt L(\lambda) = (\lambda - \Z \Delta) \cap \conv_\R \wt
L(\lambda)$ for arbitrary $\lambda \in \lie{h}^*$?
\end{qn}

This question has a positive answer; see Corollary \ref{Csimple} below.
Indeed, we go beyond the above question, in that we also describe
explicitly the set of weights $\wt L(\lambda)$ as a disjoint union of
$W_{J_\lambda}$-stable sets, where we define for $\lambda \in \lie{h}^*$:
\begin{equation}
J_\lambda := \{ i \in I : \lambda(h_i) \in \Z_+ \}.
\end{equation}

\noindent Our formulas in Corollary \ref{Csimple} specialize to the cases
of dominant integral $\lambda$, where $J_\lambda = I$, and to
antidominant $\lambda$, where $J_\lambda$ is empty -- thus, these two
families are at ``opposite ends''.

Clearly, Question \ref{Qbump} also has a positive answer for all Verma
modules.
Moreover, in Theorem \ref{Twtgvm} we will show similar formulas to those
in Corollary \ref{Csimple}, now for all Verma modules (and others). In
order to reconcile these results, we now discuss a larger family of
highest weight modules that encompasses both Verma and finite-dimensional
modules: namely, parabolic Verma modules.

\begin{defn}\label{Dgvm}\hfill
\begin{enumerate}
\item For $J \subset I$, let $\lie{g}_J$ denote the semisimple Lie
subalgebra of $\lie{g}$ generated by $\{ x_{\alpha_j}^\pm : j \in J \}$.

\item Define the parabolic Lie subalgebra $\lie{p}_J := \lie{g}_J +
\lie{h} + \lie{n}^+$ for all $J \subset I$. Now given $\lambda \in
\lie{h}^*$ and $J \subset J_\lambda$, define the corresponding
\textit{parabolic Verma module} with highest weight $\lambda$ to be
$M(\lambda,J) := U(\lie{g}) \otimes_{U(\lie{p}_J)} L_J(\lambda)$. Here,
$L_J(\lambda)$ is a simple finite-dimensional highest weight module over
the Levi subalgebra $\lie{h} + \lie{g}_J$; it is also killed by
$\lie{g}_{I \setminus J} \cap \lie{n}^+$ (in $M(\lambda, J)$).
\end{enumerate}
\end{defn}

Parabolic Verma modules thus unite Verma and simple modules as desired:
$M(\lambda,\emptyset) = M(\lambda)$ is a Verma module for all $\lambda
\in \lie{h}^*$, while if $\lambda \in P^+$, then $J_\lambda = I$ and
$M(\lambda,I)$ is the finite-dimensional simple module $L(\lambda)$.
These modules were introduced and studied by Lepowsky in a series of
papers; see \cite{Le} and the references therein. They are also called
``generalized'' or ``relative'' Verma modules; see \cite[\S 9.4]{H3}.
The following basic properties of $M(\lambda,J)$ will be used below
without reference.

\begin{theorem}[{\cite[Chapter 9]{H3}}]
Suppose $\lambda \in \lie{h}^*$ and $J \subset J_\lambda$.
\begin{enumerate}
\item $M(\lambda,J)$ is a $\lie{g}_J$-integrable $\lie{g}$-module
generated by a highest weight vector $m_\lambda$, with relations:
\[
\lie{n}^+ m_\lambda = (\ker \lambda) m_\lambda =
(x_{\alpha_j}^-)^{\lambda(h_j) + 1} m_\lambda = 0, \qquad \forall j \in
J.
\]

\item The formal character of $M(\lambda,J)$ (and hence $\wt
M(\lambda,J)$) is $W_J$-invariant.
\end{enumerate}
\end{theorem}

Given the above formulae for the supports of finite-dimensional and Verma
modules, the following question is natural.

\begin{qn}\label{Qgvm}
Can the set of weights of an arbitrary parabolic Verma module
$M(\lambda,J)$ be computed as in Question \ref{Qbump}?
\end{qn}

\noindent We answer Question \ref{Qgvm} affirmatively, for all modules
$M(\lambda,J)$ and others, in Theorem \ref{Twtgvm} below.

\subsection{Main results: 1. Convex hull of weights}

Given the positive answer to Question \ref{Qgvm}, a follow-up question is
if, more generally, $\wt \vla = (\lambda - \Z \Delta) \cap \conv_\R (\wt
\vla)$ for every highest weight module $\vla$. As we show in Theorem
\ref{Twtsimple}, this is \textit{false}. Therefore in this paper, we next
consider the ``weaker'' question of computing convex hulls of weights for
various modules $\vla$. This weaker question is relevant because convex
hulls of weights of highest weight modules are crucially used in
computing the set of weights themselves -- e.g., in Theorem \ref{Twtgvm}
below.

To proceed further, we recall the following notation.

\begin{defn}
A \textit{(convex) polyhedron} is a finite intersection of half-spaces in
Euclidean space. A \textit{(convex) polytope} is a compact polyhedron.

Given $\lambda \in P^+$, the corresponding \textit{Weyl polytope} is
defined to be $\calp(\lambda) := \conv_\R( \wt L(\lambda))$.
\end{defn}

We now discuss the convex hull of a general highest weight module $\vla$.
If $\vla = M(\lambda)$, this hull is a polyhedron with unique vertex
$\lambda$. On the other hand, if $\lambda$ is dominant integral and $\vla
= L(\lambda)$ is simple, its support $\wt L(\lambda)$ is finite and
$W$-invariant, and the corresponding Weyl polytope as well as $\wt
L(\lambda)$ are determined by Theorem \ref{Tklv}(1).
More generally, the convex hull of parabolic Verma modules is known:

\begin{prop}[Khare and Ridenour, {\cite[Proposition 2.4]{KR}}]\label{Tgvm}
Given $\lambda \in \lie{h}^*$ and $J \subset J_\lambda$, $\conv_\R \wt
M(\lambda,J)$ is a $W_J$-invariant convex polyhedron with vertices
$W_J(\lambda)$. It is the Minkowski sum of the polytope $\conv_\R
W_J(\lambda)$ and the cone $\R_+ (\Phi^- \setminus \Phi_J^-)$.
\end{prop}

However, the structure of $\conv_\R (\wt \vla)$ is not known for other
highest weight modules, not even for infinite-dimensional simple modules
$L(\lambda)$ with non-antidominant highest weights $\lambda$.
Thus, in this paper we seek to answer the following question.

\begin{qn}\label{Qhull}
Fix $\lambda \in \lie{h}^*$ and $M(\lambda) \twoheadrightarrow \vla$. Is
the convex hull $\conv_\R \wt \vla$ of the (infinite) set of weights a
convex polyhedron (i.e., cut out by \textit{finitely} many hyperplanes)?
If so, identify the vertices, extremal rays, and faces of this
polyhedron, as well as its stabilizer subgroup in $W$.
More precisely, can one write down an analogue of Proposition \ref{Tgvm}
for $\vla$, and if so, what set of simple roots plays the role of $J$, in
governing the symmetries of $\wt \vla$?
\end{qn}

This question is answered by two of our two main results. The first is as
follows.

\begin{utheorem}\label{T1}
Given $\lambda \in \lie{h}^*,\ M(\lambda) \twoheadrightarrow \vla$, and
$J \subset I$, define $\wt_J \vla := \wt \vla \cap (\lambda - \Z_+
\Delta_J)$. There exists a unique subset $J(\vla) \subset I$ such that
the following are equivalent:
(a) $J \subset J(\vla)$;
(b) $\wt_J \vla$ is finite; 
(c) $\wt_J \vla$ is $W_J$-stable;
(d) $\wt \vla$ is $W_J$-stable.
Moreover, if $\vla_\lambda$ is spanned by $v_\lambda$, then
\begin{equation}\label{Etop}
J(\vla) := \{ i \in J_\lambda : (x^-_{\alpha_i})^{\lambda(h_i) + 1}
v_\lambda = 0 \},
\end{equation}

\noindent where $J_\lambda = \{ i \in I : \lambda(h_i) \in \Z_+ \}$.
In particular, if $\vla$ is a parabolic Verma module $M(\lambda,J')$ for
$J' \subset J_\lambda$ or a simple module $L(\lambda)$, then $J(\vla) =
J'$ or $J_\lambda$ respectively. 
\end{utheorem}

\noindent (For more equivalent conditions, see Proposition \ref{Pgvm}.)
In particular, $\vla$ is finite-dimensional if and only if $I = J(\vla)$,
in which case $\lambda \in P^+$ and $\vla = M(\lambda,I) = L(\lambda)$.
More generally, Theorem \ref{T1} establishes a ``top'' part for $\vla$
that is a finite-dimensional simple module over the Levi subalgebra
$\lie{h} + \lie{g}_{J(\vla)}$. As we show in this paper and in follow-up
work \cite{Khminmax}, the distinguished set of simple roots $J(\vla)
\subset J_\lambda$ plays a crucial role in the study of the module
$\vla$.

We also show below that the subset $J(\vla)$ is closely related to the
classification theory by Fernando \cite{Fe} for simple weight modules
with finite weight multiplicities. In Proposition \ref{Phwfer}, we show
how to recover $J(\vla)$ from $\vla$, thereby reconciling our results
with those in \cite{Fe}. Fernando's results have an additional connection
to the present paper, in that outside of parabolic Verma modules, we also
study another class of highest weight modules treated by him:

\begin{defn}\label{Dpure}
A $\lie{g}$-module $M$ is \textit{pure} if for each $X \in \lie{g}$, the
set $\{ m \in M \ : \ \dim \C[X]m < \infty \}$ is either $0$ or $M$.
\end{defn}

Returning to Question \ref{Qhull}, note that it has an immediate and
positive answer for all Verma modules, and hence for all antidominant
weights $\lambda$. These weights constitute a Zariski dense set in
$\lie{h}^*$, namely the complement of countably many (affine)
hyperplanes. Thus, all ``non-Verma'' highest weight modules have highest
weights in this countable set of hyperplanes. Our next main result
completely resolves Question \ref{Qhull} for the larger set of highest
weights that avoids only the \textit{finite} set of simple root
hyperplanes. We first give these weights a name.

\begin{defn}
Define $\lambda \in \lie{h}^*$ to be \textit{simply-regular} if
$(\lambda, \alpha_i) \neq 0$ for all $i \in I$.
\end{defn}

Note that antidominant or regular weights are simply-regular, and all
simple $\lie{g}$-modules are pure \cite{Fe}. Now in stating the next
result (and henceforth), by \textit{extremal rays} at a vertex $v$ of a
polyhedron $P$, we mean the infinite length edges of $P$ that pass
through $v$.

\begin{utheorem}\label{T2}
Suppose $(\lambda, \vla)$ satisfy one of the following:
(a) $\lambda \in \lie{h}^*$ is simply-regular and $\vla$ is arbitrary;
(b) $|J_\lambda \setminus J(\vla)| \leq 1$;
(c) $\vla = M(\lambda, J')$ for some $J' \subset J_\lambda$; or
(d) $\vla$ is pure.

\noindent Then the convex hull (in Euclidean space) $\conv_\R \wt \vla
\subset \lambda + \liehr^*$ is a convex polyhedron with vertices
$W_{J(\vla)}(\lambda)$, and the stabilizer subgroup in $W$ of both $\wt
\vla$ and $\conv_\R \wt \vla$ is $W_{J(\vla)}$. If $\lambda$ is
simply-regular, the extremal rays at the vertex $\lambda$ are $\{ \lambda
- \R_+ \alpha_i \ : \ i \notin J(\vla) \}$.
\end{utheorem}

\begin{remark}
Consequently, the notion of the Weyl polytope extends to arbitrary simple
highest weight modules, via: $\calp(\lambda) := \conv_\R \wt L(\lambda) =
\conv_\R \wt M(\lambda, J_\lambda)$.
Note that one now obtains a polyhedron (which is a polytope if and only
if $\lambda \in P^+$, in which case $J(L(\lambda)) = J_\lambda = I$).
Even more generally, one can define $\calp(\vla) := \conv_\R \wt
M(\lambda, J(\vla))$; this is a $W_{J(\vla)}$-invariant convex
polyhedron, which equals $\conv_\R \wt \vla$ when $\lambda$ is
simply-regular, or $\vla = L(\lambda)$ or $M(\lambda, J')$. Moreover, its
translation by $-\lambda$ is a pseudo-Weyl lattice polyhedron, in the
spirit of \cite[\S 2.3]{Kam}.
\end{remark}

\begin{remark}
Although we work with arbitrary $\lambda \in \lie{h}^*$ (a complex vector
space), the only sets we work with in this paper are (convex hulls of)
subsets of $\wt \vla$ for various highest weight modules $M(\lambda)
\twoheadrightarrow \vla$. Thus, the convex hulls of these sets are
translates (via the highest weight) of subsets of $-\R_+ \Delta$. This
means that we essentially work in the real form $\liehr^* \cong \R^I$.
\end{remark}

\subsection{Motivation 2: faces of convex hull of weights}

Our next motivation arises out of the classification of faces of the Weyl
polytope $\calp(\lambda)$ for $\lambda \in P^+$. Weyl polytopes
$\calp(\lambda)$ were carefully studied in a series of classical papers,
including by Satake \cite[\S 2.3]{Sa} in the context of compactifications
of symmetric spaces (see also Casselman \cite[\S 3]{Cas}), the treatise
on algebraic groups by Borel--Tits \cite[\S 12]{BT}, as well as more
recently by Vinberg in \cite[\S 3]{Vi}, where he embedded
Poisson-commutative subalgebras of ${\rm Sym}(\lie{g})$ into $U(\lie{g})$
via the symmetrization map. In these works, the faces of $\calp(\lambda)$
were classified as follows:

\begin{theorem}[Satake \cite{Sa}, Borel--Tits \cite{BT}, Vinberg
\cite{Vi}, Casselman \cite{Cas}]\label{Tvin}
Given $\lambda \in P^+$, the faces of $\calp(\lambda)$ are of the form
$F_{w,J}(\lambda) := w(\conv_\R W_J(\lambda))$ with $w \in W$ and $J
\subset I$. Moreover, every face is $W$-conjugate to a unique dominant
face $F_{1,J}(\lambda)$.
\end{theorem}

These results were extended by the author and Ridenour in \cite{KR} to
all parabolic Verma modules. Then in \cite[\S 5]{CM}, Cellini and
Marietti provided an alternate, uniform description for all faces of the
\textit{root polytope} $\calp(\theta)$, where $\theta$ is the highest
root in $\Phi^+$ and $L(\theta) = \lie{g}$ is the adjoint representation.
More generally, one can ask the following:
\begin{itemize}
\item Can Theorem \ref{Tvin} be proved for all highest weight modules
$\vla$?

\item Classify the ``inclusion relations'' between faces, and in
particular, all redundancies -- i.e., which faces $F_{w,J}, F_{w',J'}$
are equal.
\end{itemize}

\noindent Partial results for the second part were shown in
\cite{BT,Cas,Sa}, but only for finite-dimensional simple modules (and
trivially for Verma modules). It was shown in \cite{KR} that $F_{1,J} =
F_{1,J'}$ in $\calp(\lambda)$ for $\lambda \in P^+$, if and only if
$W_J(\lambda) = W_{J'}(\lambda)$. 
Cellini and Marietti then resolved the second question above for the
adjoint representation (when $\lie{g}$ is simple), by showing in
\cite[Proposition 5.9]{CM} that if $\vla = L(\theta) = \lie{g}$, then
$J'$ is any subset of $I$ in some ``interval''. Namely, $\wt_{J'} \lie{g}
= \wt_J \lie{g}$ if and only if there exist $J_{\min}, J_{\max}$
depending only on $J$, such that $J_{\min} \subset J' \subset J_{\max}$.

In Theorem \ref{T3} and Corollary \ref{Cvinberg} below, and then in
follow-up work \cite{Khminmax}, we unify and extend the above results to
all highest weight modules $\vla$. In particular, we completely resolve
for most/all modules $\vla$ the first/second question above, originally
studied by Satake and others for finite-dimensional modules. We also note
that concurrent to the work in \cite{Khminmax}, Li--Cao--Li \cite{LCL}
also proved some of the results in \cite{CM}, for the special case of all
Weyl polytopes.

\subsection{Motivation 3: quantum affine algebras, combinatorics, and
weak faces}

In addition to answering longstanding questions about the structure of
simple (and other) highest weight modules, and the classification of
faces of Weyl polytopes, this paper draws from other research programs in
the literature.
Namely, we are also motivated by the study of quantum affine Lie
algebras, (multigraded) current algebras, Takiff algebras, and
cominuscule parabolics. In studying the former, one encounters an
important class of representations called Kirillov--Reshetikhin (KR)
modules \cite{kire}, which are widely studied because of their
connections to mathematical physics and their rich combinatorial
structure.

There has been tremendous activity in the literature to better understand
KR-modules. One approach is to specialize KR-modules at $q=1$; this
yields $\Z_+$-graded modules over $\lie{g} \ltimes \lie{g}$, which is a
Takiff or truncated current algebra.
As recently shown in \cite{CG3}, such specializations are projective
objects in a suitable category of $\Z_+$-graded $\lie{g} \ltimes
\lie{g}$-modules, which is constructed using a face of the root polytope
$\calp(\theta)$. This helps compute the graded characters of these
modules. Every such face also helps construct families of Koszul algebras
\cite{CG2}. This approach has been extended by the author in joint work
with Chari and Ridenour \cite{CKR} to faces of all Weyl polytopes
$\calp(\lambda)$. The results of \cite{CKR} were further extended in
\cite{BCFM}, where Chari et al used faces of Weyl polytopes to study
multigraded generalizations of KR-modules over multivariable current
algebras. Thus, understanding the faces of $\calp(\lambda)$ and the
relationships between them aids these programs as well.\medskip

There are additional questions that arise from the above program and have
been studied in the literature. In identifying KR-modules at $q=1$ as
projective objects in certain categories of $\lie{g} \ltimes
\lie{g}$-modules, Chari and Greenstein work in \cite{CG2,CG3} with a
subset $S \subset \Phi^+ = \wt \lie{n}^+$ of positive roots, which
satisfies a certain combinatorial condition. Namely, given weights
$\lambda_i \in S$ and $\mu_j \in \wt \lie{g}$,
\begin{equation}\label{ECG2}
\sum_{i=1}^r \lambda_i = \sum_{j=1}^r \mu_j \quad \implies \quad \mu_j
\in S\ \forall j.
\end{equation}

\noindent This condition arises in studying the weights of
$\bigwedge^\bullet \lie{g}$. (This is related to abelian ideals, and we
discuss connections in Remark \ref{Rabelian}.) It was shown in \cite{KR}
how Equation \eqref{ECG2} extends the notion of the face of a polytope.

We now introduce a novel tool used in this paper: that of a \textit{weak
$\bba$-face}. This was introduced in \cite{KR} for $\bba$ an arbitrary
subfield of $\R$; it is now defined in the present paper through its
characterization shown in \cite[Proposition 4.4]{KR}. We further extend
this notion to arbitrary additive subgroups $\bba \subset (\R,+)$; this
helps unify and extend results in the literature, as well as provide
common proofs.

\begin{defn}\label{Dweak}
Fix an $\R$-vector space $\V$, as well as subsets $X \subset \V$ and $R
\subset \R$.
\begin{enumerate}
\item A \textit{supporting hyperplane} to $X$ is any hyperplane
$\varphi(-) = c$, with $c \geq 0$ and $0 \neq \varphi \in \V^*$ such that
$c \in \varphi(X) \subset [0,\infty)$.

\item Define the finitely supported $R$-valued functions on $X$ to be:
\begin{equation}
\Fin(X,R) := \{ f : \V \to R \cup \{ 0 \} : \supp(f) \subset X,\ \#
\supp(f) < \infty \},
\end{equation}

\noindent where $\supp(f) := \{ v \in \V : f(v) \neq 0 \}$. Then
$\Fin(X,R) \subset \Fin(\V, \R)$ for all $X,R$.

\item Define the maps $\ell : \Fin(\V,\R) \to \R$ and $\disp :
\Fin(\V,\R) \to \R \V = \V$ via:
\begin{equation}
\ell(f) := \sum_{x \in \V} f(x), \qquad \disp(f) := \sum_{x \in \V} f(x)
x.
\end{equation}

\item We say that $Y \subset X$ is a \textit{weak $R$-face} of $X$ if for
any $f \in \Fin(X, R_+)$ and $g \in \Fin(Y,R_+)$,
\begin{equation}
\ell(f) = \ell(g) > 0,\ \disp(f) = \disp(g) \implies \supp(f) \subset Y.
\end{equation}

\item Given $X \subset \V$ (where $\V$ is a real or complex vector space)
and $\varphi \in \V^*$, define
\begin{equation}
X(\varphi) := \{ x \in X : \varphi(x) - \varphi(x') \in \R_+\ \forall x'
\in X \}
\end{equation}

\noindent to be the corresponding maximizer subset.
(Note that $\varphi$ is constant on $X(\varphi)$.)
\end{enumerate}
\end{defn}

Henceforth, the notions of polyhedra, polytopes, faces,
and supporting hyperplanes are used without reference. See \cite[\S
2.5]{KR} for definitions and results such as the Decomposition Theorem.

\begin{remark}\label{Rweak}
Weak faces generalize the notion of faces in two ways: first, if $R = \R$
and $X \subset \V$ is convex, then a weak $\R$-face is the same as a
face. Weak $R$-faces involve satisfying the same condition as (weak
$\R$-)faces, but with a different set $R_+$ of coefficients. Second, the
notion is defined and used for non-convex (in fact, discrete) subsets of
$\R^n$. Weak $R$-faces are very useful because they occur in many
settings in representation theory and convexity theory; see Remark
\ref{Rclosed}. There are also connections to distortion in geometric
group theory; see Remark \ref{Rdistort}.
\end{remark}

The following basic results on weak faces are straightforward.

\begin{lemma}\label{L0}
Suppose $Y \subset X \subset \V$, a real or complex vector space, and
$\varphi \in \V^*$. Then every nonempty subset $X(\varphi)$ is a weak
$R$-face of $X$ for all $R \subset \R$. If $\mathbb{B} \subset \R$ is a
subring, then $Y$ is a weak $\mathbb{B}$-face if and only if it is a weak
$\F(\mathbb{B})$-face, where $\F(\mathbb{B})$ is the quotient field of
$\mathbb{B}$.
\end{lemma}

Now observe that the sets $S \subset \wt \lie{g}$ satisfying Equation
\eqref{ECG2} are precisely the weak $\Z$-faces of $\wt \lie{g} = \wt
L(\theta) = \Phi \cup \{ 0 \}$ (and hence the weak $\mathbb{Q}$-faces as
well, by Lemma \ref{L0}). In joint work \cite{CKR} with Chari and
Ridenour, the results in \cite{CG2} were extended to obtain families of
Koszul algebras using weak $\mathbb{Q}$-faces of arbitrary Weyl polytopes
$\calp(\lambda)$ (as opposed to $\calp(\theta)$).

Thus, it is fruitful to understand and classify subsets $S$ satisfying
\eqref{ECG2}. Using case by case arguments, Chari et al \cite{CDR} showed
that such sets of roots $S$ are precisely the set of weights on some face
of $\calp(\theta)$. Hence one has various seemingly distinct yet related
ingredients in root polytopes: the faces of the polytope, the maximizer
subsets $(\wt \lie{g})(\xi)$, and the weak $\mathbb{Q}$-faces of $\wt
\lie{g}$. Although we observed in Remark \ref{Rweak} that weak
$\mathbb{Q}$-faces of $\wt L(\theta)$ are related to faces of
$\calp(\theta)$, one would like more precise connections between these
objects. Thus we showed with Ridenour using type-free arguments that more
generally, all of these notions are one and the same, in every Weyl
polytope. Some of our results also recover those by Satake and others
(see Theorem \ref{Tvin} above).

\begin{theorem}[Khare and Ridenour \cite{KR}; Chari et al
\cite{CDR}]\label{Tkrcdrvi}
For any $\lambda \in P^+$ and any subfield $\F$ of $\R$, the weak
$\F$-faces $S$ of $\wt L(\lambda)$ are precisely the maximizer subsets $S
= (\wt L(\lambda))(\xi)$ for some $\xi \in P$. There is a bijection
between such subsets $S$ and faces $F$ of the Weyl polytope
$\calp(\lambda)$, sending $S$ to $F = \conv_\R(S)$, or equivalently,
sending a face $F$ to $S = F \cap \wt L(\lambda)$.
\end{theorem}

\subsection{Main results: 2. Faces of convex hull of weights}

Notice that Theorem \ref{Tkrcdrvi} holds only for finite-dimensional
highest weight modules. An immediate question is if these results extend
to all modules $\vla$. Another possible extension involves working not
with a subring $\Z$ or subfield $\F$ of $\R$, but with an additive
subgroup:

\begin{qn}
Find connections as in Theorem \ref{Tkrcdrvi}, for an arbitrary highest
weight module $M(\lambda) \twoheadrightarrow \vla$, for $\lambda \in
\lie{h}^*$.
Is it also possible to classify the weak $\bba$-faces of $\wt \vla$,
where $0 \neq \bba \subset (\R, +)$ is an arbitrary nontrivial additive
subgroup? Are these equal to the sets of weights on faces of the convex
hull of weights $\conv_\R (\wt \vla)$?
\end{qn}

\noindent Our next main result completely answers these questions when
$\lambda$ is not on a simple root hyperplane (for all $\vla$), as well as
for all simple modules and parabolic Verma modules:

\begin{utheorem}\label{T3}
Suppose $(\lambda, \vla)$ satisfy one of the following:
(a) $\lambda \in \lie{h}^*$ is simply-regular and $\vla$ is arbitrary;
(b) $|J_\lambda \setminus J(\vla)| \leq 1$;
(c) $\vla = M(\lambda, J')$ for some $J' \subset J_\lambda$; or
(d) $\vla$ is pure.

\noindent The following are equivalent for a nonempty subset $Y \subset
\wt \vla$:
\begin{enumerate}
\item $Y = (\wt \vla)(\varphi)$ for some $\varphi \in \lie{h}$ (i.e., $Y$
is the set of weights on some supporting hyperplane).

\item $Y \subset \wt \vla$ is a weak $\bba$-face.

\item There exist $w \in W_{J(\vla)}$ and $J \subset I$ such that $Y =
w(\wt_J \vla)$.

If $\lambda$ is simply-regular, then these are also equivalent
to:

\item $Y \cap \wt_{J(\vla)} \vla$ is nonempty, and if $y_1 + y_2 = \mu_1
+ \mu_2$ for $y_i \in Y, \mu_i \in \wt \vla$, then $\mu_i \in Y$ as well.
\end{enumerate}
\end{utheorem}

\noindent This theorem at once characterizes and classifies all subsets
of weights that are weak $\Z$-faces (as in \cite{CDR,CG2,CKR}) or weak
$\F$-faces (as in \cite{KR}) of $\wt \vla$, or lie in the ``usual'' faces
of $\conv_\R(\wt \vla)$ (in Euclidean space, as in
\cite{BT,Cas,CM,KR,Sa,Vi}). Moreover, all of these references involve
finite-dimensional simple
modules; but these constitute a special case of our result, where
$\lambda \in P^+,\ \vla = L(\lambda)$, $J(\vla) = I$, and $\bba = \Z$ or
$\F$. In contrast, Theorem \ref{T3} holds for all simple $L(\lambda)$ as
well as all highest weight modules for simply-regular $\lambda$, for all
subgroups $\bba \subset \R$ -- and it is independent of $\bba$.

\begin{remark}
The last condition (4) in Theorem \ref{T3} is \textit{a priori} far
weaker than being a weak $\Z$-face; it was also considered by Chari et al
in \cite{CDR} for $\wt \lie{g}$. It is easy to see by Lemma \ref{Lfield}
below that there are many ``intermediate'' conditions of closedness that
are implied by (2) and imply (4) in Theorem \ref{T3}; thus, they are all
equivalent to (2) as well.
\end{remark}

\begin{remark}
Vinberg showed in \cite{Vi} that every face of the Weyl polytope
$\calp(\lambda)$ is a $W$-translate of a dominant face
$\calp(\lambda)(\mu)$ for some dominant $\mu \in \R_+ \Omega$, where
$\Omega$ is the set of fundamental weights; see also \cite[Proposition
5.1 and Theorem 5.6]{CM} for the special case of the adjoint
representation. Using Theorem \ref{T3}, it is clear how to extend this to
all $\vla$ for simply-regular $\lambda$ and to all simple $\vla =
L(\lambda)$. It is not hard to show in this case that the map $Y \mapsto
\conv_\R Y$ is a bijection from the set of weak $\Z$-faces of $\wt \vla$
to the set of faces of $\conv_\R \wt \vla$, with inverse map $F \mapsto F
\cap \wt \vla$. See Proposition \ref{Pkr2}.
\end{remark}

We also provide an ``intrinsic'' characterization of the weak faces of
$\wt \vla$ that are \textit{finite} sets, thereby generalizing results
for finite-dimensional modules $L(\lambda)$ in \cite{CDR,KR}. See Theorem
\ref{T4}.

\subsection{Other connections}

The study of Weyl polytopes -- and more generally, highest weight
modules, their structure and combinatorics -- continues to be an area of
intense activity. Early results such as the character formulas of
Weyl(-Kac) and Kostant, as well as more modern results such as
Kazhdan-Lusztig theory and the theory of crystals, have yielded direct or
algorithmic information about the characters and weights of various
simple modules.
Modern interest centers around crystal bases and canonical bases
introduced by Kashiwara and Lusztig, which are a major development in
combinatorial representation theory (see \cite{HK} and its references),
and are a widely used tool in representation theory, combinatorics, and
mathematical physics.

We present further connections to the literature. In the special case
$\lambda = \theta$, the root polytope $\calp(\theta)$ has been the focus
of much recent interest because of its importance in the study of abelian
and ad-nilpotent ideals of $\lie{g}$ (or of $\lie{b}^+$). These
connections are described in Section \ref{S4}.
Root polytopes $\calp(\theta)$ and their variants such as
$\conv_\R(\Phi^+ \cup \{ 0 \})$ have been much studied for a variety of
reasons: they are related to certain toric varieties, as discussed in
\cite[\S 1]{Chi}. Moreover, their connections to combinatorics (e.g.
computing the volumes of these polytopes, word lengths with respect to
root systems, and growth series of root lattices via triangulations) were
explored by Ardila et al, M\'esz\'aros, and even earlier by
Gelfand--Graev--Postnikov. See \cite{ABHPS,Me} and the references
therein.

Weyl polytopes also have other connections to combinatorics and
representation theory. For instance, a class of ``pseudo-Weyl polytopes''
(i.e., polytopes whose edges are parallel to roots) called
Mirkovi\'c-Vilonen (or MV) polytopes has recently been the focus of much
research. These are the image under the moment map of certain projective
varieties in the affine Grassmannian, called MV-cycles, which provide
bases of finite-dimensional simple modules over the Langlands dual group
via intersection homology. MV-cycles and polytopes (which include Weyl
polytopes) are useful in understanding weight multiplicities and tensor
product multiplicities, and also have connections to Lusztig's canonical
basis. See for instance \cite{And,Kam} for more details.

\subsection{Applications of the main results}

We conclude this section by mentioning two applications of the above
results and the methods used to prove them.
Note by Theorem \ref{T2} that $\conv_\R \wt L(\lambda) = \conv_\R \wt
M(\lambda, J_\lambda)$ for all weights $\lambda \in \lie{h}^*$. We now
return to our original motivation for analyzing the convex hull of
weights of $\lie{g}$-modules $\vla$:
\begin{itemize}
\item To compute the set of weights of all simple modules $L(\lambda)$.

\item To determine whether \cite[Theorem 7.41]{Ha} holds more generally
for all $\lambda \in \lie{h}^*$, namely:
\[
\wt L(\lambda) = (\lambda - \Z \Delta) \cap \conv_\R \wt L(\lambda).
\]
\end{itemize}

Some progress towards these and related questions is known. For instance,
as per \cite[Chapter 7]{H3}, translation functors can be used to reduce
computing the formal character of simple modules $L(\lambda)$ for
semisimple $\lie{g}$ to the principal blocks $\mathcal{O}_0$ for all
$\lie{g}$. A more involved approach uses Kazhdan-Lusztig polynomials; see
\cite{BB,BK,EW,KL} as well as \cite[Chapter 8]{H3} for more on the
subject.

We now completely resolve both questions above, by providing explicit
formulas to compute the supports of all simple modules $L(\lambda)$ and
parabolic Verma modules $M(\lambda,J')$, among others:

\begin{utheorem}\label{Twtgvm}
Suppose $\lambda \in \lie{h}^*, J' \subset J_\lambda$, and $\vla =
M(\lambda,J')$ or $M(\lambda) \twoheadrightarrow \vla$ with $|J_\lambda
\setminus J(\vla)| \leq 1$. Then,
\begin{equation}\label{Ewtgvm}
\wt \vla = (\lambda - \Z \Delta) \cap \conv_\R \wt \vla = \wt
L_{J(\vla)}(\lambda) - \Z_+ (\Phi^+ \setminus \Phi^+_{J(\vla)}) =
\bigsqcup_{\mu \in \Z_+ \Delta_{I \setminus J(\vla)}} \wt
L_{J(\vla)}(\lambda - \mu).
\end{equation}
\end{utheorem}

\noindent Theorem \ref{Twtgvm} provides three formulas for the support of
all simple or parabolic Verma modules. One of these formulas demonstrates
the invariance of $\wt L(\lambda)$ under the subgroup $W_{J_\lambda}$ of
$W$, and is a ``Verma-type'' union of finite ``integrable'' sets of
weights. (This corresponds to the integrability of $L(\lambda)$ under the
Levi subalgebra $\lie{h} + \lie{g}_{J_\lambda}$.)
Another expresses $\wt L(\lambda)$ as a ``Verma-type'' finite Minkowski
sum of rays.

To our knowledge (and that of experts), these formulas are not known in
the literature. The formulas are direct and do not involve cancellations.
(We remark that O.~Mathieu, in private communication, mentioned to us the
equation $\wt L(\lambda) = \wt M(\lambda, J_\lambda)$; note this also
follows from the formulas in Theorem \ref{Twtgvm}.) 
Also note that the last expression in Equation \eqref{Ewtgvm} corresponds
to the $\lie{g}_{J'}$-integrability of $M(\lambda,J')$, as discussed in
\cite[Chapter 9]{H3}, while the first equality extends \cite[Theorem
7.41]{Ha} to all parabolic Verma modules. We also show in Theorem
\ref{Twtsimple} that Equation \eqref{Ewtgvm} is false for general $\vla$.
However, by Theorem \ref{T1}, it holds for all simple modules:

\begin{cor}\label{Csimple}
Equation \eqref{Ewtgvm} holds upon specializing $(\vla, J(\vla))$ to
$(L(\lambda), J_\lambda)$ for all $\lambda \in \lie{h}^*$.
\end{cor}

\begin{remark}\label{Rmult}
A more ambitious goal is to compute the weight multiplicities for highest
weight modules. Note that finite-dimensional modules $\vla = L(\lambda)$
for $\lambda \in P^+$, as well as the corresponding Weyl polytopes
$\conv_\R \wt L(\lambda)$, are closely associated with
Duistermaat--Heckman functions and thus in computing the weight
multiplicities \cite{BGR,DH}.
In fact Antoine and Speiser \cite{AS} provided explicit formulas for the
characters of simple finite-dimensional $\lie{g}$-modules $L(\lambda)$ in
terms of the sets of weights $\wt L(\mu)$ for $\mu \in P^+$, for certain
low-rank simple Lie algebras $\lie{g}$. See also \cite{Bl,CT}, as well as
\cite[Theorem 3.7]{Kas}, in which Kass related characters and weight-sets
for $\{ L(\mu) : \mu \in P^+ \}$ and provided a recursive formula to
compute the character of $L(\lambda)$. (Translation functors are also
used in character computations; see \cite[Chapter 7]{H3}.) The present
paper has limited results regarding weight multiplicities; see Theorem
\ref{Twcf}, where we extend the Weyl Character Formula to simple modules
$L(\lambda)$ for highest weights $\lambda$ which are not necessarily
dominant integral. Also note that the $J(\vla)$-integrability (or the
convex hull) of $\vla$ can be used to compute the multiplicity of
specific weights $\mu$ of $\vla$ (or to show it is nonzero), using the
multiplicity of other weights in the $W_{J(\vla)}$-orbit of $\mu$ which
may be easier to compute. We will explore in future work, how the methods
and results of the present paper can be used to obtain further
information involving weight multiplicities.
\end{remark}

Our second application is ``dual'' to Theorem \ref{T1} in the following
sense: Theorem \ref{T1} identifies a ``largest parabolic subgroup of
symmetries'' given a highest weight module. Dually, it is possible to
identify a largest and a smallest highest weight module, given a
parabolic group of symmetries.

\begin{utheorem}\label{Tminmax}
Fix $\lambda \in \lie{h}^*$ and $J' \subset  J_\lambda$ such that either
$\lambda$ is simply-regular, or $J' = \emptyset$ or $J_\lambda$. There
exist unique ``largest'' and ``smallest'' highest weight modules
$M_{\max}(\lambda, J'), M_{\min}(\lambda,J')$ such that the following are
equivalent for a nonzero highest weight module $M(\lambda)
\twoheadrightarrow \vla$:
\begin{enumerate}
\item $\conv_\R \wt \vla = \conv_\R \wt M(\lambda,J')$.
\item The stabilizer subgroup in $W$ of $\conv_\R \wt \vla$ is $W_{J'}$.
\item The largest parabolic subgroup of $W$ that preserves $\conv_\R \wt
\vla$ is $W_{J'}$.
\item $M(\lambda) \twoheadrightarrow M_{\max}(\lambda,J')
\twoheadrightarrow \vla \twoheadrightarrow M_{\min}(\lambda,J')$.
\end{enumerate}
\end{utheorem}

The results and techniques in this paper yield further rewards. For
instance, in related work \cite{Khminmax}, we extend very recent
combinatorial results by Cellini and Marietti \cite{CM} on the adjoint
representation of a simple Lie algebra $\lie{g}$, to arbitrary highest
weight modules $\vla$ over (finite-dimensional) semisimple $\lie{g}$. In
Corollary \ref{Cvinberg}, we also completely classify all inclusion
relations between faces of all highest weight modules $\vla$ with generic
highest weight (in fact, any simply-regular weight) $\lambda \in
\lie{h}^*$. This extends Theorem \ref{Tvin} and related results by
Satake, Borel--Tits, Vinberg, and Casselman to large classes of
infinite-dimensional modules $\vla$. A full classification of inclusion
relations between faces for all weights $\lambda$ is more involved, and
is obtained in follow-up work \cite{Khminmax}.

\begin{remark}
Since this paper went to press, we have been able to strengthen some of
the above results. Namely, we show in a forthcoming paper with G.~Dhillon
\cite{DK} that Theorems \ref{T2}, \ref{T3}, \ref{Tminmax} hold for
\textit{all} highest weight modules, as well as more generally, over
arbitrary Kac-Moody algebras $\lie{g}$.
\end{remark}

\section{Classifying (positive) weak faces for simply-regular highest
weights}\label{S4}


The remainder of this paper is devoted to proving the results stated in
Section \ref{S2}.
In the present section, we study (weak) faces of $\conv_\R \wt \vla$ for
all modules $\vla$ with simply-regular $\lambda$, i.e., satisfying
$\lambda(h_i) \neq 0\ \forall i \in I$. This provides alternate proofs of
the main results of \cite{KR,Vi}, which were known to date only for
finite-dimensional $\vla$.
The proofs in this section are algebraic/combinatorial. Thus they differ
from previous papers and future sections in that they are case-free as
opposed to the case-by-case analysis in \cite{CDR}, and use neither the
Decomposition Theorem for convex polyhedra as in \cite{KR}, nor the
geometry of the Weyl group action as in \cite{Vi}.

In this section we consider several combinatorial conditions among
subsets of $\wt \lie{g}$ (some not yet mentioned in this paper), which
were studied by Chari and her co-authors in \cite{CDR,CG2}, as well as in
joint works \cite{CKR,KR} by the author. To state these conditions for
general $\vla$, some additional notation is needed.

\begin{defn}
Let $X$ be a subset of a real vector space, and $R \subset \R$ be any
(nonempty) subset.
\begin{enumerate}
\item $Y \subset X$ is a \textit{positive weak $R$-face} if for any $f
\in \Fin(X, R_+)$ and $g \in \Fin(Y,R_+)$,
\begin{equation}
\disp(f) = \disp(g) \implies \ell(g) \leq \ell(f),
\end{equation}

\noindent with equality if and only if $\supp(f) \subset Y$. Note that
this definition is consistent with the notation and results in \cite{KR},
via \cite[Proposition 4.4]{KR}.

\item Given $R,R' \subset \R$, we say that $Y \subset X$ is
\textit{$(R',R)$-closed} if given $f \in \Fin(X, R),\ g \in \Fin(Y,R)$,
\begin{equation}
\ell(f) = \ell(g) \in R' \setminus \{ 0 \},\ \disp(f) = \disp(g) \implies
\supp(f) \subset Y.
\end{equation}

\item Define the \textit{$R$-convex hull} of $X$ to be the image under
$\disp$ of $\{ f \in \Fin(X, R \cap [0,1]) : \ell(f) = 1 \}$. This will
be denoted by $\conv_R(X)$.
\end{enumerate}
\end{defn}

\noindent (Positive) weak $\Z$-faces were studied and used in
\cite{CDR,CG2,KR}. Weak $R$-faces are the same as $(\R, R_+)$-closed
subsets. Moreover, the ``usual'' convex hull of a set $X$ is simply
$\conv_\R(X)$.

The goal of this section is to prove Theorem \ref{T3} when $\lambda$ is
simply-regular. More precisely, we classify the (positive) weak faces of
$\wt \vla$ that contain the vertex $\lambda$. Later, Theorem \ref{T2}
will help prove Theorem \ref{T3} without the restriction of containing
$\lambda$. Here is the main result in this section.

\begin{theorem}\label{Tface}
Given $\lambda \in \lie{h}^*$ and $M(\lambda) \twoheadrightarrow \vla$
with highest weight space $\vla_\lambda = \C v_\lambda$,
\[
\wt_J \vla = (\wt \vla)(\rho_{I \setminus J}) = \wt U(\lie{g}_J)
v_\lambda\ \forall J \subset I.
\]

\noindent Now fix an additive subgroup $0 \neq \bba \subset (\R,+)$, and
a subset $Y \subset \wt \vla$ that contains $\lambda$. Then each part
implies the next:
\begin{enumerate}
\item There exists a subset $J \subset I$ such that $Y = \wt_J \vla$.

\item $Y = (\wt \vla)(\varphi)$ for some $\varphi \in \lie{h}$.

\item $Y$ is a weak $\bba$-face of $\wt \vla$.

\item $Y$ is $(\{ 2 \}, \{ 1, 2 \})$-closed in $\wt \vla$.
\end{enumerate}

\noindent If $\lambda - \alpha_i \in \wt \vla$ for all $i \in I$, then
these are all equivalent, and the $J$ in part (1) is unique.
\end{theorem}

\noindent Thus we are able to classify the weak $\bba$-faces of $\wt
\vla$ that contain $\lambda$, if $\lambda - \Delta \subset \wt \vla$. By
Corollary \ref{Cprev} below, this holds for all $\vla$ if $\lambda$ is
simply-regular. Moreover, the result shows that the weak $\bba$-faces of
$\wt \vla$ containing $\lambda$ can be described independently of $\bba$.

For completeness, we also classify which of these weak $\bba$-faces (from
Theorem \ref{Tface}) are positive weak $\bba$-faces. (By Proposition
\ref{Pkr} below, every positive weak $\bba$-face is necessarily a weak
$\bba$-face.)

\begin{theorem}\label{Trigid}
Fix $\lambda \in \lie{h}^*,\ J \subset I$, and an additive subgroup $0
\neq \bba \subset (\R,+)$. Then $\wt_J \vla$ is a positive weak
$\bba$-face of $\wt \vla$ if exactly one of the following occurs:
\begin{itemize}
\item $\lambda \notin \bba \Delta$ and $J \subset I$ is arbitrary, or

\item $\lambda \in \bba \Delta$, and there exists $j_0 \notin J$ such
that $(\lambda, \omega_{j_0}) > 0$.
\end{itemize}

\noindent The converse holds if $\lambda - \alpha_i \in \wt \vla\ \forall
i \in I$ and $a \bba \subset \bba$ for some $0 \neq a \in \bba$ (e.g.,
$\Z \cap \bba \neq 0$).
\end{theorem}

\noindent Thus while the weak $\bba$-faces of $\wt \vla$ are independent
of $\bba$, the same cannot be said of the positive weak $\bba$-faces.

\subsection{Basic properties of closedness}

We now develop some preliminaries in order to prove Theorems \ref{Tface}
and \ref{Trigid}.
In \cite{CDR,CG2}, Chari et al discuss various combinatorial conditions,
and study the sets of roots in $\wt \lie{g} = \Phi \cup \{ 0 \}$ that
satisfy these conditions. These include the condition of being a weak
$\Z$-face as well as of being a positive weak $\Z$-face (which were
subsequently studied in all Weyl polytopes in \cite{CKR}). Another result
from \cite{CDR} is as follows:\medskip

\noindent \textit{``A proper subset $Y \subset \Phi^+$ is a weak
$\Z$-face if and only if: $\alpha + \beta, \alpha + \beta - \gamma \notin
\Phi\ \forall \alpha,\beta \in Y, \gamma \in \Phi \setminus
Y$.''}\medskip

\noindent In other words, $(Y + Y) \cap \Phi = (Y + Y) \cap [\Phi + (\Phi
\setminus Y)] = \emptyset$. It is natural to ask how to extend this
condition to arbitrary modules $\vla$. To do so, note that $0 \in \wt
L(\theta) \setminus Y$, so that the above condition is equivalent to the
following:
\[
(Y + Y) \cap (\wt \lie{g} + \{ 0 \}) = (Y + Y) \cap [\Phi + (\Phi
\setminus Y)] = \emptyset.
\]

\noindent In other words, $0 \notin Y \subset \wt \lie{g}$ is $(\{ 2 \},
\{ 1, 2 \})$-closed. In Theorems \ref{T3} and \ref{Tface}, we study this
condition in a general highest weight module.

\begin{remark}\label{Rclosed}
The notion of $(R',R_+)$-closedness thus occurs in the literature for
various $R',R \subset \R$:
\begin{itemize}
\item $R = \F$ and $R' \supset \F_+$ for a subfield $\F \subset \R$ (as
in weak $\F$-faces in \cite{KR}).

\item $R = \Z$ and $R' \supset \Z_+$; this is used in
\cite{CDR,CG2,CKR,KR}.

\item We address all of these (above) cases by working in greater
generality in this paper, with $R = \bba$ and $R' \supset \bba_+$ for an
additive subgroup $0 \neq \bba \subset (\R,+)$ (as in weak $\bba$-faces).

\item $R = R' = \R$ occurs in convexity theory and linear programming,
when one works with faces of polytopes and polyhedra, which are precisely
intersections with supporting hyperplanes.

\item $R' = \{ 2 \}$ and $R = \{ 1, 2 \}$ or $\{ 0, 1, 2 \}$ (as in in
\cite{CDR}).
\end{itemize}
\end{remark}

\begin{remark}\label{Rabelian}
Another combinatorial condition involves subsets $\Psi \subset \Phi^+$
that satisfy:
\begin{equation}\label{Esuter}
(\Psi + \Psi) \cap \Phi = \emptyset, \qquad (\Psi + \Phi^+) \cap \Phi
\subset \Psi.
\end{equation}

\noindent Such subsets $\Psi$ are precisely the \textit{abelian ideals}
of $\Phi^+$. Abelian and ad-nilpotent ideals connect affine Lie
algebras/Weyl groups, the algebra $\bigwedge^\bullet \lie{g}$ of
Maurer--Cartan left-invariant differential forms, combinatorial
conditions on sets of roots, and other areas.
Recent interest in abelian ideals can be traced back to the seminal work
of Kostant (and Peterson) \cite{Ko2} where he showed that abelian ideals
were intricately connected to Cartan decompositions and discrete series.
They have since attracted much attention, including by Cellini--Papi
\cite{CP}, Chari--Dolbin--Ridenour \cite{CDR}, Panyushev \cite{Pa} (and
R\"ohrle \cite{PR}), and Suter \cite{Su}.

Although we do not discuss further connections to abelian ideals in this
paper, we remark that they have several combinatorial properties, such as
the characterization via Equation \eqref{Esuter}. Kostant showed in
\cite[Theorem 7]{Ko1} that the map sending an abelian ideal $\Psi$ to
$\sum_{\mu \in \Psi} \mu \in P$ is one-to-one and yields the highest
weights of certain irreducible summands of the finite-dimensional
$\lie{g}$-module $\bigwedge^\bullet \lie{g}$.
Moreover, it is easy to check that Equation \eqref{Esuter} is satisfied
by all subsets $\wt_J L(\theta)$ for $J \subsetneq I$. In particular, the
abelian ideal $\wt_J L(\theta)$ was denoted in \cite{CDR} by $\lie{i}_0$
and is the unique ``minimal'' ad-nilpotent ideal in the corresponding
parabolic Lie subalgebra $\lie{p}_J$ of $\lie{g}$.
\end{remark}

We now present a few basic results on (positive) weak faces and
closedness, which are used to prove Theorems \ref{Tface} and
\ref{Trigid}. The following are straightforward by using the definitions.

\begin{lemma}\label{Lfield}
Fix subsets $R,R' \subset \R$ and $0 < a \in \R$. Suppose $Y \subset X
\subset \V$, a real vector space.
\begin{enumerate}
\item If $Y \subset X$ is $(R',R)$-closed and $X_1 \subset X$ is
nonempty, then $Y \cap X_1 \subset X_1$ is $(R'_1, R_1)$-closed, where
$R'_1 \subset a \cdot R'$ and $R_1 \subset a \cdot R$.

\item For any $v \in \V$, $Y \subset X$ is $(R',R)$-closed if and only if
$v \pm aY \subset v \pm aX$ is $(R',R)$-closed.

\item For all $\varphi \in \V^*$, $X(\varphi)$ is $(R',R_+)$-closed in
$X$ for all $R,R' \subset \R$.

\item Given an invertible linear transformation $T \in GL(\V)$, $T(Y)
\subset T(X)$ is a (positive) weak $R$-face or $(R',R)$-closed, if and
only if $Y \subset X$ is a weak $R$-face or $(R',R)$-closed,
respectively.

\item If $\varphi(x) \in (0,\infty)$ for some $x \in X$, then
$X(\varphi)$ is a positive weak $R$-face of $X$.
\end{enumerate}
\end{lemma}

Next, if $R = R' = \F_+$ for a subfield $\F \subset \R$, then results in
\cite{KR} relate weak $\F$-faces and positive weak $\F$-faces. We now
show this more generally (and add another equivalent condition) for
$\bba$.

\begin{prop}\label{Pkr}
Fix $Y \subset X \subset \V$ (a real vector space) and an additive
subgroup $0 \neq \bba \subset (\R,+)$. The following are equivalent:
\begin{enumerate}
\item $Y$ is a positive weak $\bba$-face of $X$.

\item $0 \notin Y$, and $Y$ is a weak $\bba$-face of $X \cup \{ 0 \}$ -
i.e.,
\begin{eqnarray*}
&& \sum_{x \in X} a_x x + c \cdot 0 = \sum_{y \in Y} b_y y \in \bba_+ X
\bigcap \bba_+ Y,\ a_x, b_y, c \in \bba_+\ \forall x,y, \quad c + \sum_x
a_x = \sum_y b_y\\
& \implies & c = 0, \quad x \in Y \mbox{ if } a_x > 0.
\end{eqnarray*}

\item $Y$ is a weak $\bba$-face of $X$ and $X \cup \{ 0 \}$; $0$ is not a
nontrivial $\bba_+$-linear combination of $Y$.
\end{enumerate}
\end{prop}

\noindent If $1 \in \bba$, then the last part of (3) can be replaced by:
$0 \notin \conv_{\bba}(Y)$; the proof would be similar.

\begin{proof}
We prove a cyclic chain of implications. First assume (1), and choose $0
< a \in \bba$. If $0 \in Y$, then define $f(0) = a,\ g(0) = 2a$, and
$f(x) = g(x) = 0\ \forall x \in \V \setminus \{ 0 \}$. Then $\disp(f) = 0
= \disp(g)$, but $\ell(f) = a < 2a = \ell(g)$, which contradicts the
definitions. Hence $0 \notin Y$.
Now suppose $\disp(f) = \disp(g)$ and $\ell(f) = \ell(g)$ for $f \in
\Fin(X \cup \{ 0 \}, \bba_+)$ and $g \in \Fin(Y, \bba_+)$. Set $f_1 := f$
on $X \setminus \{ 0 \}$ and $f_1(0) := 0$; then $\disp(f_1) = \disp(f) =
\disp(g)$, but $\ell(f_1) \leq \ell(f) = \ell(g)$. By (1), $\ell(f_1) =
\ell(g) = \ell(f)$ and $\supp(f_1) \subset Y$. But then $f(0) = 0$,
whence $f \equiv f_1$ and $\supp(f) \subset Y$ as well. This proves (2).

Now assume (2). Since $Y$ is a weak $\bba$-face of $X \cup \{ 0 \}$ and
$Y \subset X$, hence $Y$ is a weak $\bba$-face of $X$ from the
definitions. It remains to show that $0 \neq \disp(f)$ for any $0 \neq f
\in \Fin(Y, \bba_+)$. Suppose otherwise; then $0 = \sum_i r_i y_i$, where
(finitely many) $0 < r_i \in \bba$, and $y_i \in Y$ are pairwise
distinct. Now define $f(0) := \sum_i r_i$ and $g(y_i) := r_i$ for all $i$
(and $f,g$ are $0$ at all other points). Then $\disp(f) = 0 = \disp(g)$
and $\ell(f) = \sum_i r_i = \ell(g)$, so $\supp(f) = \{ 0 \} \subset Y$,
which is a contradiction.

Finally, we show that $(3) \implies (1)$. Suppose $\disp(f) = \disp(g)$
for $f \in \Fin(X, \bba_+)$ and $g \in \Fin(Y, \bba_+)$. If $\ell(g) >
\ell(f)$, then define $f_1(0) := f(0) + \ell(g) - \ell(f)$, and $f_1 :=
f$ otherwise.
Then $\disp(f_1) = \disp(f) = \disp(g)$, and $\ell(f_1) = \ell(g)$. Since
$Y \subset X \cup \{ 0 \}$ is a weak $\bba$-face, hence $\supp(f_1)
\subset Y$. But then $0 \in Y$. Now choose $0 < a \in \bba$; then $0  = a
\cdot 0$ is a nontrivial $\bba_+$-linear combination of $Y$. This is a
contradiction, so $\ell(g) \leq \ell(f)$.
Now suppose $\ell(g) = \ell(f)$; since $Y \subset X$ is a weak
$\bba$-face, hence $\supp(f) \subset Y$ as desired. Conversely, if
$\supp(f) \subset Y$, then define $f_1(0) := f(0) + \ell(f) - \ell(g)$,
and $f_1 := f$ otherwise. Now $\disp(f_1) = \disp(f) = \disp(g)$ and
$\ell(f_1) = \ell(g)$. Since $Y \subset X \cup \{ 0 \}$ is a weak
$\bba$-face, hence $\supp(f_1) \subset Y$. Moreover, $0 = a \cdot 0
\notin Y$ by assumption (for any $0 < a \in \bba$). Hence $f_1(0) = 0$,
whence $\ell(f) = \ell(g)$ (and $f(0) = 0$), and (1) is proved.
\end{proof}

\begin{remark}\label{Rdistort}
We briefly digress to explain the choice of notation $\ell, \disp$, as
well as a connection to geometric group theory. Let $G$ be an abelian
group and $X \subset G$ a set of generators. The associated Cayley graph
is the quiver $Q_X(G)$ with set of vertices $G$, and edges $g \to gx$ for
all $g \in G,\ x \in X$. Similarly one defines $Q_X(G)$ for all $X
\subset G$.

Given $g,h \in G$ and $X \subset G$, let $\calp_X(g,h)$ be the set of
paths in $Q_X(G)$ from $g$ to $h$, and let $\calp_X^n(g,h)$ be the subset
of paths of length $n$. One can then define the same notions: $\ell :
\Fin(X,\Z_+) \to \Z_+$ and $\disp : \Fin(X,\Z_+) \to G$ in this setting
as well. Now $\ell, \disp$ act on paths, as long as they are considered
to be finite sets of edges together with multiplicities. (Note that one
can add them in any order, since $G$ is assumed to be abelian.) It is now
clear that $\ell$ takes such a path to its ``$X$-length'', and $\disp$ to
the ``displacement'' in $G$. This explains the choice of notation.

We now reinterpret the notions of (positive) weak $\Z$-faces of $X$.
Given $Y \subset X \subset G$, it is easy to see that $Y$ is a weak
$\Z$-face of $X$ if and only if for all $n>0$,
\[
\calp_Y^n(g,h) \neq \emptyset \implies \calp_X^n(g,h) = \calp_Y^n(g,h),
\]

\noindent and $Y$ is a positive weak $\Z$-face of $X$ if and only if $Y$
``detects geodesics'':
\[
\calp_Y(g,h) \neq \emptyset \implies \calp_X^{min}(g,h) = \calp_Y(g,h),
\]

\noindent where $\calp_X^{min}(g,h)$ is the set of geodesics (i.e., paths
of minimal length) from $g$ to $h$ in $Q_X(G)$. In particular, note that
all paths in $Q_Y(G)$ (i.e., in $\calp_Y(g,h)$) must have the same
length. Moreover, if $Y$ is a positive weak face of $X$ and $X$ is
finite, the subgroup $\langle Y \rangle$ is \textit{undistorted} in
$\langle X \rangle$ (see e.g.~\cite{Gr}). In other words, the distortion
function of $\langle Y \rangle$ in $\langle X \rangle$ is
$\Delta^{\langle X \rangle}_{\langle Y \rangle}(n) \approx n$.
\end{remark}

\subsection{Proof of the results}

We now show Theorems \ref{Tface} and \ref{Trigid}. To do so, a better
understanding of the sets $\wt_J \vla$ is needed. Also recall the
following standard notation: a weight vector $m \in M_\lambda$ in a
$\lie{g}$-module $M$ is \textit{maximal} if $\lie{n}^+ m = 0$. In this
case $M(\lambda)$ maps into $M$.

\begin{lemma}\label{Lweights}
Suppose $M(\lambda) \twoheadrightarrow \vla$ (with highest weight space
$\C v_\lambda$) and $\mu \in \wt_J \vla $, for some $\lambda \in
\lie{h}^*$ and $J \subset I$. Then there exist $\mu_j \in \wt_J \vla$
such that (in the standard partial order on $\lie h^*$,)
\[
\lambda = \mu_0 > \mu_1 > \dots > \mu_N = \mu, \qquad \mu_j - \mu_{j+1}
\in \Delta_J\ \forall j, \qquad N \geq 0.
\]

\noindent Moreover, if $\vla = L(\lambda)$ is simple, then so is the
$\lie{g}_J$-submodule $\vla_J := U(\lie{g}_J) v_\lambda$.
\end{lemma}

\noindent In fact, it turns out that a more general phenomenon is true;
see Theorem \ref{Tkumar} in the appendix.

\begin{proof}
Given $\mu \in \wt_J \vla$, $0 \neq \vla_\mu = U(\lie{n}^-)_{\mu -
\lambda} v_\lambda$, and every such weight vector in $U(\lie{n}^-)$ is a
linear combination of Lie words generated by the $x_{\alpha_i}^-$ (with
$\alpha_i \in \Delta$). Hence there is some $f$ in the subalgebra $R :=
\C \langle \{ x_{\alpha_i}^- \} \rangle$ of $U(\lie{n}^-) \subset
U(\lie{g})$, such that $f v_\lambda \neq 0$. Writing $f$ as a $\C$-linear
combination of monomial words (each of weight $\mu - \lambda$) in this
image $R$ of the free algebra on $\{ x_{\alpha_i}^- : i \in I \}$, at
least one such monomial word $x^-_{\alpha_{i_N}} \cdots
x^-_{\alpha_{i_2}} x^-_{\alpha_{i_1}}$ does not kill $v_\lambda$ (with
$i_j \in I\ \forall j$). Hence $\mu_j := \wt (x^-_{\alpha_{i_j}}
x^-_{\alpha_{i_{j-1}} } \cdots x^-_{\alpha_{i_1}} v_\lambda) \in \wt
\vla$ for all $j$.
Since $\mu \in \wt_J \vla \subset \lambda - \Z_+ \Delta_J$ and $\Delta$
is a basis of $\lie{h}^*$, hence $\mu_j \in \wt_J \vla$ and $\mu_j -
\mu_{j+1} = \alpha_{i_{j+1}} \in \Delta_J$ for all $j < N$. This shows
the first part.

To show the second statement, suppose $\vla_J$ is not a simple
$\lie{g}_J$-module.
Define $\lie{n}^\pm_J$ to be the Lie subalgebra generated by $\{
x_{\alpha_j}^\pm : j \in J \}$. Then there exists a maximal vector
$v_\mu$ (with $\mu \neq \lambda$) in the weight space $(\vla_J)_\mu =
U(\lie{n}^-_J)_{\mu - \lambda} v_\lambda$, which is killed by all of
$\lie{n}^+_J$. By the Serre relations, $v_\mu$ is also a maximal vector
in $\vla$, since $\lie{n}^+_{I \setminus J}$ commutes with $\lie{n}^-_J$.
Since $\mu \neq \lambda$, $\vla$ is not simple either.
\end{proof}

We now show the main result in this section.

\begin{proof}[Proof of Theorem \ref{Tface}]
Define $\vla_J$ as in Lemma \ref{Lweights}. Then one inclusion for the
first claim is obvious: $\wt \vla_J \subset \wt_J \vla$.
Conversely, given $\mu \in \wt_J \vla$, the proof of Lemma \ref{Lweights}
implies that $\vla_\mu$ is spanned by monomial words in $\lie{n}^-_J$
applied to $v_\lambda$. In particular, $\mu \in \wt \vla_J$, as desired.
Now $\wt_J \vla$ is contained in $\lambda - \Z_+ \Delta_J$, and
$\rho_{I \setminus J} \in P^+$. This easily shows that if $\mu \in \wt
\vla \subset \wt M(\lambda)$, then $(\lambda, \rho_{I \setminus J}) -
(\mu, \rho_{I \setminus J}) \in \Z_+$, with equality if and only if $\mu
\in \lambda - \Z_+ \Delta_J$. Thus,
\begin{equation}\label{Esub}
\wt U(\lie{g}_J) v_\lambda = \wt \vla_J = \wt_J \vla = (\wt \vla)(\rho_{I
\setminus J}).
\end{equation}

Next, clearly $(1) \implies (2) \implies (3) \implies (4)$ by Equation
\eqref{Esub} and Lemma \ref{Lfield} (dividing by any $0 < a \in \bba$).
Now assume (4), as well as that $\lambda - \alpha_i \in \wt \vla\ \forall
i \in I$. Define $J := \{ i \in I : \lambda - \alpha_i \in Y \}$. We
claim that $Y = \wt_J \vla$, which proves (1). To see the claim, first
suppose that $\mu \in \wt_J \vla$. By Lemma \ref{Lweights}, there exist
$\mu_0 = \lambda > \mu_1 > \dots > \mu_N = \mu$ such that $\mu_{i-1} -
\mu_i = \alpha_{l_i} \in \Delta_J$ for all $1 \leq i \leq N$. Then $l_i
\in J$ and $\lambda - \alpha_{l_i} \in Y$ for all $i$. We claim that $\mu
= \mu_N \in Y$ by induction on $N$. First, this is true for $N=0,1$ by
assumption. Now if $\mu_0, \dots, \mu_{k-1} \in Y$, then
\[
\mu_0 + \mu_k = \mu_{k-1} + (\lambda - \alpha_{l_k}).
\]

\noindent Since both terms on the right are in $Y$, and $Y$ is $(\{ 2 \},
\{ 1, 2 \})$-closed in $X$, hence so are the terms on the left, and the
claim follows by induction. This proves one inclusion: $\wt_J \vla
\subset Y$.

Now choose any weight $\mu = \lambda - \sum_{i \in I} n_i \alpha_i \in
Y$. Again by Lemma \ref{Lweights}, there exist weights $\mu_0 = \lambda >
\mu_1 > \dots > \mu_N = \mu$ with $\mu_{i-1} - \mu_i = \alpha_{l_i}$ for
some $l_i \in I$. The next step is to show that all $\mu_i \in Y$ and all
$l_i \in J$, by downward induction on $i$. To begin, $\mu_{N-1} +
(\lambda - \alpha_{l_N}) = \mu_0 + \mu_N = \lambda + \mu$. Since both
terms on the right are in $Y$, so are the terms on the left. Continue by
induction, as above.
This argument shows that if $n_i > 0$ for any $i$ (in the definition of
$\mu$ above), then $\lambda - \alpha_i \in Y$, so $i \in J$. But then
$\mu = \lambda - \sum_{i \ : \ n_i > 0} n_i \alpha_i \in \wt_J \vla$, as
desired.
\end{proof}

We conclude this part by showing the remaining unproved result in this
section.

\begin{proof}[Proof of Theorem \ref{Trigid}]
In this proof, we repeatedly use Proposition \ref{Pkr} without
necessarily referring to it henceforth. Set $Y := \wt_J \vla \subset X =
\wt \vla$.

First suppose that $\lambda \notin \bba \Delta$, and $J \subset I$ is
arbitrary. One easily checks that $0 \notin \wt_J \vla$, so it suffices
to show that $\wt_J \vla$ is a weak $\bba$-face of $\{ 0 \} \cup \wt
\vla$. Suppose $\sum_{y \in Y} m_y y = \sum_{x \in X} r_x x + (\sum_y m_y
- \sum_x r_x) 0$, with $\sum_y m_y \geq \sum_x r_x$ and all $m_y, r_x \in
\bba_+$. Then we have:
\[
\sum_y m_y (\lambda - y) = \sum_x r_x (\lambda - x) + (\sum_y m_y -
\sum_x r_x) \lambda.
\]

\noindent The left side is in $\bba_+ \Delta_J$, whence so is the right
side. Now $\lambda - x \in \Z_+ \Delta$ and $\lambda \notin \bba \Delta$,
so by the independence of $\Delta$, $\sum_y m_y = \sum_x r_x$ and
$\lambda - x \in \Z_+ \Delta_J$ whenever $r_x > 0$. In particular, $\wt_J
\vla$ is a weak $\bba$-face of $\{ 0 \} \cup \wt \vla$, and we are done
by Proposition \ref{Pkr}.

If $\lambda \in \bba \Delta$ instead, fix $j_0 \notin J$ such that
$(\lambda, \omega_{j_0}) > 0$. For all $\mu = \lambda - \sum_{i \in J}
a_i \alpha_i \in \wt_J \vla$, we have $(\mu, \omega_{j_0}) = (\lambda,
\omega_{j_0}) > 0$ by assumption, so $0 \notin \wt_J \vla$. Now say
$\sum_i a_i (\lambda - \mu_i) = \sum_j b_j (\lambda - \beta_j) + c \cdot
0$ and $\sum_i a_i = \sum_j b_j + c$ for $a_i, b_j, c \in \bba_+, \mu_i
\in \Z_+ \Delta_J, \beta_j \in \Z_+ \Delta$. Taking the inner product with
$\omega_{j_0}$,
\[
D \sum_i a_i = D \sum_j b_j - \sum_j b_j (\beta_j, \omega_{j_0}) \leq D
\sum_j b_j,
\]

\noindent where $D = (\lambda, \omega_{j_0}) > 0$. Dividing, $\sum_i a_i
\leq \sum_j b_j = \sum_i a_i - c \leq \sum_i a_i$, whence the two sums
are equal and $c=0$. Thus $\sum_j b_j \beta_j = \sum_i a_i \mu_i \in
\bba_+ \Delta_J$, whence $\beta_j \in \Z_+ \Delta_J\ \forall j$.
Therefore $Y = \wt_J \vla$ is a weak $\bba$-face of $\{ 0 \} \cup \wt
\vla$, and $Y$ is a positive weak $\bba$-face of $\wt \vla$ by
Proposition \ref{Pkr}.

Now assume that $\lambda - \alpha_i \in \wt \vla\ \forall i$. To show the
(contrapositive of the) converse, write $\lambda = \sum_{i \in I_+} c_i
\alpha_i - \sum_{j \in I_-} d_j \alpha_j$, where $c_i, d_j \in \bba_+$
and $I_\pm := \{ i \in I : \pm (\lambda, \omega_i) > 0 \}$. Then for $r
\in \R$,
\[
\left( r + \sum_{j \in I_-} r d_j \right) \lambda + \sum_{i \in I_+} r
c_i (\lambda - \alpha_i) = \sum_{i \in I_+} r c_i \cdot \lambda + \sum_{j
\in I_-} r d_j (\lambda - \alpha_j).
\]

\noindent The weights on the left are in $Y$, since $I_+ \subset J$. Now
choose $0 < r := |a| \in \bba$ as in the assumptions. Then the
coefficients on the left side add up to $|a|(1 + \sum_{i \in I_+} c_i +
\sum_{j \in I_-} d_j)$, which is larger than the sum of the right-hand
coefficients. Hence $Y$ is not a positive weak $\bba$-face of $\wt \vla$.
\end{proof}

\begin{remark}
The above proof also shows that $\wt_J \vla$ is not a positive weak
$\bba$-face of $\wt \vla$ if $\Z \cap \bba \neq 0$ and $\lambda \in \bba
\Delta_{J \setminus J_0} - \bba_+ \Delta_{I \setminus (J \cup J_0)}$,
where $J_0 := \{ i \in I : \lambda(h_i) = 0 \}$. This is because if $i
\notin J_0$, then $\lambda - \alpha_i \in \wt L(\lambda) \subset \wt
\vla$.
\end{remark}

\subsection{Connection to previous work}

We now show how Theorems \ref{Tface} and \ref{Trigid} provide alternate
proofs of results in previous papers, and hold for all highest weight
modules $\vla$ for ``generic'' $\lambda$.

\begin{cor}\label{Cprev}
Fix $\lambda \in \lie{h}^*$, $M(\lambda) \twoheadrightarrow \vla$, and an
additive subgroup $0 \neq \bba \subset (\R,+)$. Then
Theorems \ref{Tface} and \ref{Trigid} classify:
\begin{enumerate}
\item all (positive) weak $\bba$-faces of $\wt \vla$ containing
$\lambda$, if $\lambda$ is simply-regular.

\item all (positive) weak $\bba$-faces of $\wt \vla$, if $\lambda - \Z_+
\alpha_i \subset \wt \vla$ for all $i \in I$.

\item all $(\{ 2 \}, \{ 1, 2 \})$-closed subsets of $\wt \vla$, if $\vla
= M(\lambda)$.
\end{enumerate}
\end{cor}

\noindent In this result, to classify the positive weak $\bba$-faces, we
also assume that $1 \in \bba$.

\begin{proof}
If $\lambda - \alpha_i \in \wt \vla$ for all $i \in I$, then every weak
$\bba$-face containing $\lambda$ is of the form $\wt_J \vla$ for some $J
\subset I$, by Theorem \ref{Tface}. Hence so is every positive weak
$\bba$-face containing $\lambda$ (by the definitions, or by Proposition
\ref{Pkr}); now Theorem \ref{Trigid} classifies all the positive weak
$\bba$-faces.

Next, suppose $\lambda \in \lie{h}^*$ is simply-regular and $\vla$ is
arbitrary. It suffices to prove that $\lambda - \alpha_i \in \wt \vla$
for all $i \in I$; this holds if we show it for the irreducible quotient
$L(\lambda)$ of $\vla$. Now compute:
\[
x_{\alpha_i}^+ (x_{\alpha_i}^- v_\lambda) = h_{\alpha_i} v_\lambda =
(2(\lambda, \alpha_i) / (\alpha_i, \alpha_i)) v_\lambda,
\]

\noindent and this is nonzero for all $i \in I$ because $\lambda$ is
simply-regular. This implies that $x_{\alpha_i}^- v_\lambda$ is nonzero
in $L(\lambda)$, which proves the claim for $L(\lambda)$, and hence for
$\vla$.

Now assume that $\lambda$ is arbitrary and $\lambda - \Z_+ \alpha_i
\subset \wt \vla$ for all $i \in I$. If $\mu := \lambda - \sum_{i \in I}
n_i \alpha_i \in Y$, then $\displaystyle (1 + |I|) \mu = \lambda +
\sum_{i \in I} \left( \lambda - (1 + |I|)n_i \alpha_i \right)$. Hence
$\lambda \in Y$ since $Y \subset \wt \vla$ is a weak $\bba$-face.
Finally, suppose $Y \subset \wt M(\lambda)$ is $(\{ 2 \}, \{ 1, 2
\})$-closed (e.g., a weak $\bba$-face). If $y = \lambda - \sum_{i \in I}
n_i \alpha_i \in Y$, then $\lambda + (\lambda - \sum_i 2 n_i \alpha_i) =
y + y$, so $\lambda \in Y$, as claimed. But now $Y = \wt_J \vla$ by
Theorem \ref{Tface}.
\end{proof}

We end this section by extending Theorem \ref{Tvin} by Vinberg and
others, from finite-dimensional modules, to arbitrary highest weight
modules $\vla$ for simply-regular $\lambda$.

\begin{cor}\label{Cvinberg}
Suppose $\lambda$ is simply-regular: $\lambda(h_i) \neq 0\ \forall i \in
I$. Also fix an arbitrary highest weight module $M(\lambda)
\twoheadrightarrow \vla$. Then the {\em face map}
$\mathscr{F}^{(1)}_{\vla}$, which sends a set $J \subset I$ of simple
roots to the corresponding weak face $\wt_J \vla$ of $\wt \vla$, is
one-to-one.
\end{cor}

\noindent The result is an immediate corollary of the last assertion in
Theorem \ref{Tface}, and provides a complete classification of all
inclusion relations between the (weak) faces of $\wt \vla$.

\section{Finite maximizer subsets and parabolic Verma
modules}\label{S5}

The rest of this paper is devoted to proving the main theorems stated in
Section \ref{S2}. In this section, we analyze in detail the weak
$\bba$-faces $\wt_J \vla$ that are finite, and thus prove Theorem
\ref{T1}. We first introduce and study an important tool needed here and
below: the maps $\varpi_J$.

\begin{defn}
Given $J \subset I$, define $\pi_J : \lie{h}^* = \C \Omega_I
\twoheadrightarrow \C \Omega_J$ to be the projection map with kernel $\C
\Omega_{I \setminus J}$, where for any subset $J \subset I$, $\Omega_J$
denotes the set of fundamental weights $\{ \omega_j : j \in J \}$.
Also define $\varpi_J : \lambda + \C \Delta_J \to \pi_J(\lambda) + \C
\Delta_J$ (where the codomain comes from $\lie{g}_J$) as follows:
$\varpi_J(\lambda + \mu) := \pi_J(\lambda) + \mu$.
\end{defn}

\begin{remark}\label{Rfacts}
Observe that for all $\lambda \in \lie{h}^*$ and $J \subset I$,
$\pi_J(\lambda) = \sum_{j \in J} \lambda(h_j) \omega_j$.
Moreover, for all $\lambda$ and $J$, $\pi_J(\lambda)(h_i)$ equals
$\lambda(h_i)$ or $0$, depending on whether or not $i \in J$.
\end{remark}

\begin{lemma}\label{Lfacts}
Suppose $\lambda \in \lie{h}^*$ and $J \subset I$. Also fix a highest
weight module $M(\lambda) \twoheadrightarrow \vla$, with highest weight
vector $0 \neq v_\lambda \in \vla_\lambda$.
\begin{enumerate}
\item $J \subset J_\lambda$ if and only if $\pi_J(\lambda) \in P^+$ (in
fact, in $P_J^+$).

\item Let $\vla_J := U(\lie{g}_J) v_\lambda$. Then for all $J,J' \subset
I$, $(\lambda - \Z_+ \Delta_{J'}) \cap \vla_J = \wt_{J \cap J'} \vla$.

\item $\vla_J$ is a highest weight $\lie{g}_J$-module with highest weight
$\pi_J(\lambda)$. In other words, $M_J(\pi_J(\lambda)) \twoheadrightarrow
U(\lie{g}_J) v_\lambda$, where $M_J$ denotes the corresponding Verma
$\lie{g}_J$-module.

\item For all $w \in W_J$ and $\mu \in \C \Delta_J$,
$w(\varpi_J(\lambda + \mu)) = \varpi_J(w(\lambda + \mu))$.
\end{enumerate}
\end{lemma}

\noindent For all $\vla$ and $J \subset J(\vla)$, $\varpi_J$
identifies some weights of the highest weight $\lie{g}$-module $\vla$
with those of a finite-dimensional simple $\lie{g}_J$-module. More
precisely,
$\varpi_J : \wt_J \vla \to L_J(\pi_J(\lambda))$ is a bijection.

\begin{proof}
(1) follows from the definitions; (2) from the linear independence of
$\Delta$ and Equation \eqref{Esub}; and (3) from Equation \eqref{Efacts}
and Remark \ref{Rfacts}. Finally for (4), note that the computation of $w
\mu$ in either $\lie{g}$ or $\lie{g}_J$ yields the same answer in $\C
\Delta_J$, since it only depends on the root (sub)system $\Phi_J$ and the
corresponding Dynkin (sub)diagram. Thus, we can set $\mu = 0$ and prove
the result by induction on the length $\ell(w) = \ell_J(w)$ of $w \in
W_J$, the base case of $\ell(w) = 0$ being obvious. Now say the statement
holds for $w \in W$, and write: $w(\lambda) = \lambda - \nu$, with $\nu
\in \C \Delta_J$. Given $j \in J$,
\begin{eqnarray*}
(s_j w)(\varpi_J(\lambda)) & = & s_j \varpi_J(w(\lambda)) = s_j
\varpi_J(\lambda - \nu) = s_j(\pi_J(\lambda) - \nu) = \pi_J(\lambda) -
\pi_J(\lambda)(h_j) \alpha_j - s_j(\nu),\\
s_j(w(\lambda)) & = & s_j(\lambda - \nu)
= \lambda - \lambda(h_j) \alpha_j - s_j(\nu).
\end{eqnarray*}

\noindent But $\lambda(h_j) = \pi_J(\lambda)(h_j)$ by Remark
\ref{Rfacts}, and as above, the computation of $s_j(\nu)$ in either
setting is the same. Hence $\varpi_J(s_j(w(\lambda))) = (s_j
w)(\varpi_J(\lambda))$ and the proof is complete by induction.
\end{proof}

\subsection{The finite-dimensional ``top'' of a highest weight module}

The heart of this section is in the following result -- and it
immediately implies much of Theorem \ref{T1}.

\begin{prop}\label{Pgvm}
Fix $\lambda \in \lie{h}^*,\ J \subset I$, and a highest weight module
$M(\lambda) \twoheadrightarrow \vla$ with highest weight vector $0 \neq
v_\lambda \in \vla_\lambda$. Then the following are equivalent:
\begin{enumerate}
\item $J \subset J_\lambda$ and $M(\lambda) \twoheadrightarrow
M(\lambda,J) \twoheadrightarrow \vla$.

\item $J \subset J_\lambda$ and $\vla_J := U(\lie{g}_J) v_\lambda \cong
L_J(\pi_J(\lambda))$, the simple highest weight $\lie{g}_J$-module.

\item $\dim U(\lie{g}_J) v_\lambda < \infty$.

\item $\wt_J \vla$ is finite.

\item $\wt_J \vla$ is $W_J$-stable.
\end{enumerate}
\end{prop}

\noindent To show the result, use Equation \eqref{Esub}, Remark
\ref{Rfacts}, and Lemma \ref{Lfacts}, to prove the implications:
\[
(1) \implies (2) \implies (3) \implies (4) \implies (3) \implies (2)
\implies (1) \Longleftarrow (5) \Longleftarrow (2).
\]

\noindent As the arguments involved are mostly standard, we omit this
proof for brevity.

Using Proposition \ref{Pgvm}, it follows that every highest weight module
has a ``finite-dimensional top'':

\begin{proof}[Proof of Theorem \ref{T1}]
Given $\lambda$ and $\vla$, define $J(\vla)$ via Equation \eqref{Etop}.
We first show that the conditions in Proposition \ref{Pgvm} are all
equivalent to: $J \subset J(\vla)$. By definition, $M(\lambda,J(\vla))
\twoheadrightarrow \vla$, so for all $J \subset J(\vla)$, $M(\lambda,J)
\twoheadrightarrow \vla$. Hence $\wt_J \vla$ is finite by Proposition
\ref{Pgvm}. Conversely, by that same result, if $\wt_J \vla$ is finite
for any $J$, then $J \subset J_\lambda$ and $M(\lambda,J)
\twoheadrightarrow \vla$, so
$(x_{\alpha_j}^-)^{\lambda(h_j) + 1} v_\lambda = 0\ \forall j \in J$. Now
$J \subset J(\vla)$ as claimed.

For the equivalences, it remains to show that $\wt \vla$ is $W_J$-stable
if and only if $J \subset J(\vla)$. Let $\lie{p}$ denote the parabolic
Lie subalgebra $\lie{p}_{J(\vla)}$ as in Definition \ref{Dgvm}(2). Then
\cite[Lemma 9.3, Proposition 9.3, and Theorem 9.4]{H3} imply that
$M(\lambda,J(\vla)) \in \mathcal{O}^{\lie{p}}$, so $\vla \in \mathcal{O}$
lies in $\mathcal{O}^{\lie{p}}$ as well, and $\wt \vla$ is stable under
$W_{J(\vla)}$.
Now if $i \notin J(\vla)$, then since $\vla_{\lambda - n \alpha_i} = \C
\cdot (x^-_{\alpha_i})^n v_\lambda$ for all $n \geq 0$, it follows that
$s_i$ does not preserve the root string $\lambda - \Z_+ \alpha_i = (\wt
\vla) \cap (\lambda - \Z \alpha_i)$. Hence $s_i$ does not preserve $\wt
\vla$.

Finally, if $\vla = M(\lambda,J')$ for $J' \subset J_\lambda$ and $i
\notin J'$, then $\lambda - \Z_+ \alpha_i \subset \wt \vla$ by
\cite[Proposition 2.3]{KR}. By the above analysis, $J(\vla) \subset J'$.
Since $U(\lie{g}_{J'}) m_\lambda \subset \vla$ is finite-dimensional,
hence $J' = J(\vla)$.
Now recall that for all $i \in I$ and $n \geq 0$, the Kostant partition
function yields: $\dim M(\lambda)_{\lambda - n \alpha_i} = 1$. If $\vla =
L(\lambda)$ is simple and $j \in J_\lambda$, then
$(x^-_{\alpha_j})^{\lambda(h_j) + 1} m_\lambda \in M(\lambda)$ is a
maximal vector in $M(\lambda)$, whence $\wt_{\{ j \}} L(\lambda)$ is
finite if $j \in J_\lambda$. It is also easy to see by highest weight
$\lie{sl}_2$-theory that $\wt_{\{ j \}} L(\lambda) = \lambda - \Z_+
\alpha_j$ if $j \notin J_\lambda$. Hence $J(L(\lambda)) = J_\lambda$ from
above.
\end{proof}

\subsection{Characterizing finite weak faces}

We conclude this section by characterizing all finite weak faces of
highest weight modules $\vla$, of the form $\wt_J \vla$. To state the
result, we need some notation.

\begin{defn}\label{Dsupp}
Recall that the \textit{support} of a weight $\lambda \in \lie{h}^*$ is
$\supp(\lambda) := \{ i \in I : (\lambda, \alpha_i) \neq 0 \}$.
\begin{enumerate}
\item Given $J \subset I$, define $C(\lambda,J) \subset J$ to be the set
of nodes in the connected graph components of the Dynkin (sub)diagram of
$J \subset I$, which are not disjoint from $\supp(\lambda)$.

\item Given $X \subset \lie{h}^*$, define $\chi_X$ to be the indicator
function of $X$, i.e., $\chi_X(x) := 1_{x \in X}$.

\item Given a finite subset $X \subset \lie{h}^*$, define $\rho_X :=
\sum_{x \in X} x = \disp(\chi_X)$.
\end{enumerate}
\end{defn}

Chari et al showed in \cite{CDR} that $S \subset \wt \lie{g}$ is a weak
$\Z$-face of $\wt L(\theta)$ if and only if $S = (\wt \lie{g})(\rho_S)$.
Thus, a natural question is if similar ``intrinsic'' characterizations
exist for general highest weight modules. It turns out that finite weak
$\Z$-faces $S \subset \wt \vla$ are indeed characterized by $\rho_S =
\sum_{y \in S} y$ for all $\vla$. Additionally, they are also uniquely
determined by $\ell$ and $\disp$:

\begin{theorem}\label{T4}
Given $\lambda \in \lie{h}^*$ and $M(\lambda) \twoheadrightarrow \vla$,
fix $w \in W$ that preserves $\wt \vla$. Given $J \subset J(\vla)$ and a
finite subset $S \subset \wt \vla$, $S = w(\wt_J \vla)$ if and only if
$\ell(\chi_S) = \ell(\chi_{w(\wt_J \vla)})$ and $\disp(\chi_S) =
\disp(\chi_{w(\wt_J \vla)})$. Moreover, the following equality of
maximizer subsets holds:
\begin{equation}
\wt_J \vla = (\wt \vla)(\rho_{I \setminus J}) = (\wt_{J(\vla)}
\vla)(\pi_{J(\vla)} \rho_{\wt_J \vla}).
\end{equation}
\end{theorem}

In order to prove Theorem \ref{T4}, we collect together some results from
\cite{KR}.

\begin{theorem}[Khare and Ridenour, \cite{KR}]\label{Tkr1}
Fix $0 \neq \lambda \in P^+$, a subfield $\F \subset \R$, and a nonempty
proper subset $Y \subsetneq \wtvla{}$. Now define $\rho_Y := \sum_{y \in
Y} y$. The following are equivalent:
\begin{enumerate}
\item There exist $w \in W$ and $C(\lambda,I) \nsubseteq J$ such that $wY
= \wtvla{J}$.

\item $Y$ is a positive weak $\F$-face of $\wtvla{}$.

\item $Y$ is a weak $\F$-face of $\wtvla{}$.

\item $Y$ is the maximizer in $\wtvla{}$ of the functional $(\rho_Y, -)$,
with maximum value $(\rho_Y, \rho_Y) / |Y|$.

\item $Y$ is the maximizer in $\wtvla{}$ of a nonzero linear functional.
\end{enumerate}

\noindent Moreover, $\rho_{\wtvla{J}} \in P^+$ for all $J \subset I$.
\end{theorem}

\noindent More generally, one can consider (positive) weak $\bba$-faces
for any additive subgroup $0 \neq \bba \subset (\R,+)$. It is not hard to
show that these are also equivalent to the notions in Theorem \ref{Tkr1}:

\begin{cor}\label{Ckr1}
Setting as in Theorem \ref{Tkr1}. Also fix a subgroup $0 \neq \bba
\subset (\R, +)$. Then $Y \subsetneq \wtvla{}$ is a weak $\F$-face of
$\wtvla{}$ if and only if $Y \subsetneq \wtvla{}$ is a (positive) weak
$\bba$-face.
\end{cor}

\begin{proof}
If $Y$ is a weak $\F$-face, then by Theorem \ref{Tkr1}, $Y =
(\wtvla{})(\rho_Y)$ and $\rho_Y(Y) > 0$. Hence $Y$ is a positive weak
$\bba$-face of $\wtvla{}$ by Lemma \ref{Lfield}, hence a weak $\bba$-face
by Proposition \ref{Pkr}. Conversely, suppose $Y$ is a weak $\bba$-face
of $\wtvla{}$. Choosing $0 < a \in \bba$, it is easy to see by Lemmas
\ref{L0} and \ref{Lfield} that $Y \subset \wtvla{}$ is a weak $a\Z$-face,
hence a weak $\Z$-face and a weak $\mathbb{Q}$-face as well. Now by
Theorem \ref{Tkr1}, $Y = (\wtvla{})(\varphi)$ for some $\varphi$, and
hence also a weak $\F$-face of $\wtvla{}$.
\end{proof}

\noindent To prove Theorem \ref{T4}, we need one last proposition.

\begin{prop}\label{Pstable}
Fix $\lambda \in \lie{h}^*,\ M(\lambda) \twoheadrightarrow \vla$, and $J
\subset J(\vla)$.
\begin{enumerate}
\item Then $\rho_{\wt_J \vla}$ is $W_J$-invariant, and in $P^+_{J_\lambda
\setminus J} \times \C \Omega_{I \setminus J_\lambda}$.

\item Define $\rho_{I \setminus J} := \sum_{i \notin J} \omega_i$. Then
(notation as in Lemma \ref{L0} and Remark \ref{Rfacts}) for all $J'
\subset J_\lambda$:
\begin{equation}\label{Efinsub}
\wt_J \vla = (\wt \vla)(\rho_{I \setminus J}) = (\wt_{J(\vla)}
\vla)(\pi_{J(\vla)} \rho_{\wt_J \vla}) \subset (\wt \vla)(\pi_{J'}
\rho_{\wt_J \vla})
\end{equation}

\noindent and $0 \leq (\pi_{J'} \rho_{\wt_J \vla})(\wt_J \vla) \in \Z_+$.
\end{enumerate}
\end{prop}

\noindent As a consequence of the first part, $(\rho_{\wt_J \vla},
\alpha_j) = 0$ for all $j \in J$, if $\wt_J \vla$ is finite.

\begin{proof}\hfill
\begin{enumerate}
\item By Proposition \ref{Pgvm}, $\wt_J \vla$ and hence $\rho_{\wt_J
\vla}$ is $W_J$-stable. Thus it is fixed by each $s_j$ for $j \in J$, so
$(\rho_{\wt_J \vla}, \alpha_j) = 0\ \forall j \in J$. Next,
$\lambda(h_j), -\alpha_{j'}(h_j) \in \Z_+$ for $j \in J_\lambda$ and $j'
\neq j$ in $I$. Hence for each $\mu \in \wt_J \vla \subset \lambda - \Z_+
\Delta_J$, $\mu(h_j) \in \Z_+$ if $j \in J_\lambda \setminus J$. Thus,
$\rho_{\wt_J \vla}(h_j) \in \Z_+$ as well, so $\rho_{\wt_J \vla} \in
P^+_{J_\lambda \setminus J} \times \C \Omega_{I \setminus J_\lambda}$.

\item The first equality is from Theorem \ref{Tface}.
Now given $J' \subset J_\lambda$, $\pi_{J'} \rho_{\wt_J \vla} \in P^+_{J'
\setminus J} \subset P^+_{J_\lambda}$ by the previous part. Hence by
definition of $J_\lambda$, $(\pi_{J'} \rho_{\wt_J \vla}, \lambda) \in
\Z_+$, and by the previous sentence, $(\pi_{J'} \rho_{\wt_J \vla},
\alpha_j) = 0\ \forall j \in J$. Thus the linear functional $(\pi_{J'}
\rho_{\wt_J \vla}, -)$ is constant on $\wt_J \vla$, and the value is in
$\Z_+$.
Moreover, given any $\alpha \in \Delta$, $(\pi_{J'} \rho_{\wt_J \vla},
\alpha) \in \Z_+$, so the linear functional can never attain strictly
larger values than at $\lambda$. This proves the inclusion.

Now $\pi_{J'} \rho_{\wt_J \vla} \in P^+_{J_\lambda} = \Z_+
\Omega_{J_\lambda}$, so $(\pi_{J'} \rho_{\wt_J \vla}, \lambda) \in \Z_+$
by the definition of $J_\lambda$. The inclusion now implies the
inequality.
To show the second equality, note by Proposition \ref{Pgvm} that
$\vla_{J(\vla)} \cong M := L_{J(\vla)}(\pi_{J(\vla)}(\lambda))$ as
$\lie{g}_{J(\vla)}$-modules. Thus $M$ is a finite-dimensional simple
$\lie{g}_{J(\vla)}$-module, and the bijection $\varpi_{J(\vla)} :
\wt_{J(\vla)} \vla \to \wt M$ (from Proposition \ref{Pgvm}) sends
$\lambda - \nu \in \wt_{J(\vla)} \vla$ to $\pi_{J(\vla)}(\lambda) - \nu
\in \wt M$. Moreover, for all $j \in J(\vla)$, the two weights agree at
$h_j$. Now for all $j \in J(\vla)$, Remark \ref{Rfacts} implies that
\begin{equation}\label{Emax}
\pi_{J(\vla)}(\rho_{\wt_J \vla})(h_j) = \rho_{\wt_J \vla}(h_j) =
\sum_{\mu \in \wt_J \vla} \mu(h_j) = \sum_{\mu \in \wt_J \vla}
\varpi_{J(\vla)}(\mu)(h_j) = \rho_{\wt_J M}(h_j).
\end{equation}

\noindent Hence $\pi_{J(\vla)}(\rho_{\wt_J \vla}) = \rho_{\wt_J M}$ as
elements of $P^+_{J(\vla) \setminus J} \subset P^+_{J(\vla)}$.
Now the inclusion shown earlier in this part, for $J' = J(\vla)$, proves
that $\wt_J \vla \subset T := (\wt_{J(\vla)} \vla)(\pi_{J(\vla)}
\rho_{\wt_J \vla})$. Conversely, suppose  $\lambda - \nu \in T$, with
$\nu \in \Z_+ \Delta_{J(\vla)}$. Then $(\pi_{J(\vla)} \rho_{\wt_J \vla},
\nu) = 0$ since $\lambda \in T$, so $(\rho_{\wt_J M}, \nu) = 0$ by
Equation \eqref{Emax}.
Moreover, $\pi_{J(\vla)}(\lambda) - \nu \in \wt M$ (via the bijection
$\varpi_{J(\vla)}$). Therefore $\pi_{J(\vla)}(\lambda) - \nu \in (\wt
M)(\rho_{\wt_J M}) = \wt_J M$ (by Theorem \ref{Tkr1} for
$\lie{g}_{J(\vla)}$). This implies that $\nu \in \Z_+ \Delta_J$, whence
$\lambda - \nu \in \wt_J \vla$ as required.
\end{enumerate}
\end{proof}

It is now possible to characterize the weak $\bba$-faces of $\wt
\vla$ that are finite sets.

\begin{proof}[Proof of Theorem \ref{T4}]
The last equation was shown in Proposition \ref{Pstable} and Theorem
\ref{Tface} (this latter holds for all $J \subset I$).
For the first equivalence, one implication is obvious. For the converse,
$\lambda - \mu \in \Z_+ \Delta\ \forall \mu \in \wt \vla$, whence
$(\rho_{I \setminus J}, \lambda - \mu) \geq 0$. Equality is attained if
and only if $\lambda - \mu \in \Z_+ \Delta_J$ (i.e., $\mu \in \wt_J
\vla$). Thus given any finite subset $S \subset \wt \vla$, compute using
the assumptions:
\begin{eqnarray*}
0 & \leq & \sum_{\mu \in S} (\rho_{I \setminus J}, \lambda - w^{-1}(\mu))
= \left( \rho_{I \setminus J}, \sum_{\mu \in  S} (\lambda - w^{-1}(\mu))
\right) = (\rho_{I \setminus J}, \ell(\chi_S) \lambda -
w^{-1}(\disp(\chi_S)))\\
& = & (\rho_{I \setminus J}, \ell(\chi_{w(\wt_J \vla)}) \lambda -
w^{-1}(\disp(\chi_{w(\wt_J \vla)}))) = (\rho_{I \setminus J},
\ell(\chi_{\wt_J \vla}) \lambda - \disp(\chi_{\wt_J \vla}))\\
& = & \sum_{\mu \in \wt_J \vla} (\rho_{I \setminus J}, \lambda - \mu) =
0.
\end{eqnarray*}

\noindent Thus, the inequality is actually an equality, which means that
$w^{-1}(S) \subset \wt_J \vla$ by the above analysis. Since $|w^{-1}(S)|
= \ell(\chi_S) = \ell(\chi_{\wt_J \vla}) = |\wt_J \vla|$, hence
$w^{-1}(S) = \wt_J \vla$.
\end{proof}

\section{Application 1: Weights of simple highest weight
modules}\label{Sappl1}

In this section, we use the above results and techniques to compute the
support of various highest weight modules. We then discuss the more
involved question of computing the weight multiplicities in $L(\lambda)$;
see Theorem \ref{Twcf}.

\begin{proof}[Proof of Theorem \ref{Twtgvm}]
Note that if $\vla = M(\lambda,J')$, then the first and third expressions
in Equation \eqref{Ewtgvm} are equal by \cite[\S 9.4]{H3}. We now show a
cyclic chain of inclusions:
\[
\wt M(\lambda,J') \subset (\lambda - \Z \Delta) \cap \conv_\R \wt
M(\lambda, J') \subset \bigsqcup_{\mu \in \Z_+ \Delta_{I \setminus J'}}
\wt L_{J'}(\lambda - \mu) \subset \wt M(\lambda,J').
\]

The first inclusion is obvious since $\wt M(\lambda,J')$ is contained in
each factor. Also note that the last expression in Equation
\eqref{Ewtgvm} is indeed a disjoint union since $\Delta$ is a basis of
$\lie{h}^*$. Now to show the third inclusion, first note that $\lambda -
\mu \in P^+_{J'} \cap \wt M(\lambda,J')$ for all $\mu \in \Z_+ \Delta_{I
\setminus J'}$. Moreover, if $0 \neq m_{\lambda - \mu} \in
M(\lambda,J')_{\lambda - \mu}$, then it is easy to verify that
$m_{\lambda - \mu}$ is killed by all $x^+_{\alpha_j}$ for $j \in J'$.
Hence
\[
\wt L_{J'}(\lambda - \mu) = \wt U(\lie{g}_{J'}) m_{\lambda - \mu} \subset
\wt M(\lambda,J') \quad \forall \mu \in \Z_+ \Delta_{I \setminus J'},
\]

\noindent and the third inclusion follows.
Next, we show the second inclusion. Since $\conv_\R \wt M(\lambda,J')$
$\subset \conv_\R \wt M(\lambda) = \lambda - \R_+ \Delta$, it suffices to
show that
\begin{equation}\label{Eincl}
(\lambda - \Z \Delta) \cap \conv_\R \wt M(\lambda, J') = (\lambda - \Z_+
\Delta) \cap \conv_\R \wt M(\lambda,J') \subset \bigsqcup_{\mu \in \Z_+
\Delta_{I \setminus J'}} \wt L_{J'}(\lambda - \mu).
\end{equation}

Now suppose $\lambda - \nu$ is in (the intersection on) the left-hand
side of Equation \eqref{Eincl}, where $\nu = \sum_{i \in I} n_i \alpha_i
\in \Z_+ \Delta$. Since both sides of Equation \eqref{Eincl} are
$W_{J'}$-stable, there exists $w \in W_{J'}$ such that $w(\lambda - \nu)
\in P^+_{J'} \times \R \Omega_{I \setminus J'}$. Now set $\mu := \sum_{i
\notin J'} n_i \alpha_i$; then by the $W_{J'}$-invariance of the
left side,
\[
w(\lambda - \nu) \in (\lambda - \mu) - Q^+_{J'} = (\lambda - \mu) - \Z_+
\Delta_{J'},
\]

\noindent and both of these are weights in $P^+_{J'} \times \R \Omega_{I
\setminus J'}$. Hence by Theorem \ref{Tklv},
\[
\conv_\R W_{J'}(\lambda - \nu) \subset \conv_\R W_{J'}(\lambda - \mu) =
\conv_\R \wt L_{J'}(\lambda - \mu).
\]

\noindent Consequently, using Theorem \ref{Tklv},
\[
w(\lambda - \nu) \in (\lambda - \mu - Q^+_{J'}) \cap \conv_\R
W_{J'}(\lambda - \nu) \subset (\lambda - \mu - Q^+_{J'}) \cap \conv_\R
\wt L_{J'}(\lambda - \mu) = \wt L_{J'}(\lambda - \mu).
\]

\noindent Thus $\lambda - \nu \in \wt L_{J'}(\lambda - \mu)$, which shows
Equation \eqref{Eincl}. Equation \eqref{Ewtgvm} now follows for
$M(\lambda,J')$.\medskip

Next, given a general highest weight module $\vla$, Theorem \ref{T1} and
Proposition \ref{Pgvm} show that $M(\lambda, J(\vla)) \twoheadrightarrow
\vla$, whence $\wt \vla \subset \wt M(\lambda, J(\vla))$. Now
\textbf{claim} that $\wt \vla = \wt M(\lambda, J(\vla))$. To see this,
note that $\vla$ is $\lie{g}_{J(\vla)}$-integrable by results in
\cite{H3} (as discussed in the proof of Theorem \ref{T1}). Hence by
Equation \eqref{Ewtgvm} for $M(\lambda, J(\vla))$ (and the proof of the
third inclusion above), it suffices to show that $\lambda - \Z_+
\Delta_{I \setminus J(\vla)} \subset \wt \vla$ if $|J_\lambda \setminus
J(\vla)| \leq 1$.
Thus, suppose $J_\lambda \setminus J(\vla) \subset \{ i_0 \}$ for some
$i_0 \in I$. We now obtain a contradiction by assuming that there exists
$\mu = \sum_{i \notin J(\vla)} n_i \alpha_i \in \Z_+ \Delta_{I \setminus
J(\vla)}$ such that $\lambda - \mu \notin \wt \vla$. Indeed, choose such
a weight $\mu$ of minimal height $\sum_{i \notin J(\vla)} n_i$. Then
$\vla_{\lambda - \mu} = 0$, so if $v_\lambda$ spans $\vla_\lambda$, then
the following vector is zero in $\vla$:
\[
v_{\lambda - \mu} := (x^-_{\alpha_{i_0}})^{{\bf 1}(i_0 \notin J(\vla))
\cdot n_{i_0}} \cdot \prod_{i \notin J(\vla) \cup \{ i_0 \}}
(x^-_{\alpha_i})^{n_i} \cdot v_\lambda
\]

\noindent under some enumeration of $I \setminus (J(\vla) \cup \{ i_0 \})
= \{ i_1, \dots, i_m \}$. But then applying powers of
$x^+_{\alpha_{i_j}}$ for $1 \leq j \leq m$ still yields zero. Now compute
inductively, using $\lie{sl}_2$-theory and the Serre relations:
\begin{align*}
0 = \prod_{j=1}^m (x^+_{\alpha_{i_j}})^{n_{i_j}} \cdot v_{\lambda - \mu}
= &\ (x^-_{\alpha_{i_0}})^{{\bf 1}(i_0 \notin J(\vla)) \cdot n_{i_0}}
\cdot \prod_{j=1}^m n_{i_j}! \prod_{k=1}^{n_{i_j}} (\lambda(h_{i_j}) - k
+ 1) \cdot v_\lambda\\
\in &\ \C^\times (x^-_{\alpha_{i_0}})^{{\bf 1}(i_0 \notin J(\vla)) \cdot
n_{i_0}} \cdot v_\lambda.
\end{align*}

However, if $i_0 \notin J(\vla)$, then (using the Kostant partition
function,) $\lambda - \Z_+ \alpha_{i_0} \subset \wt \vla$.
This yields a contradiction, so no such $\mu$ exists and the claim is
proved. Equation \eqref{Ewtgvm} now follows easily for $\vla$.
\end{proof}

Given Theorem \ref{Twtgvm}, it is natural to ask if Equation
\eqref{Ewtgvm} holds more generally for other highest weight modules
$\vla$. We now show that this is false.

\begin{theorem}\label{Twtsimple}
Fix $\lambda \in \lie{h}^*$, $M(\lambda) \twoheadrightarrow \vla$, and
$J' \subset J_\lambda$. If $|J_\lambda \setminus J'| \leq 1$, then
\begin{equation}\label{Ewtsimple}
J(\vla) = J' \quad \implies \quad \wt \vla = (\lambda - \Z \Delta) \cap
\conv_\R \wt \vla.
\end{equation}

\noindent However, Equation \eqref{Ewtsimple} need not always hold if
$|J_\lambda \setminus J'| = 2$; and if $|J_\lambda \setminus J'| \geq 3$,
then Equation \eqref{Ewtsimple} always fails to hold for some $\vla$ with
$J(\vla) = J'$.
\end{theorem}

\begin{remark}
Thus, the formula for $\wt \vla$ may not always be as ``clean'' as the
formula for its convex hull. For instance, if $\lambda$ is
simply-regular, then the convex hull of $\wt \vla$ was computed in
Theorem \ref{T2} (and depends only on $J(\vla)$). However, the set $\wt
\vla$ need not satisfy Equation \eqref{Ewtsimple}. Thus by Equation
\eqref{Ewtgvm}, $\wt \vla$ need not always equal $\bigsqcup_{\mu \in \Z_+
\Delta_{I \setminus J(\vla)}} \wt L_{J(\vla)}(\lambda - \mu)$.

Moreover, an obvious consequence of Theorem \ref{Twtsimple} is that the
convex hull $\conv_\R \wt \vla$ does not uniquely determine the module
$\vla$ (or even its set of weights).
\end{remark}

\begin{proof}[Proof of Theorem \ref{Twtsimple}]
Equation \eqref{Ewtsimple} follows from Theorem \ref{Twtgvm} when
$|J_\lambda \setminus J'| \leq 1$. We now claim that if $s_i s_j = s_j
s_i \in W$ for some simple reflections corresponding to $i \neq j \in
J_\lambda \setminus J'$, then Equation \eqref{Ewtsimple} fails for some
$\vla$ with $J(\vla) = J'$. This shows the remaining assertions in the
theorem, since no Dynkin diagram of finite type contains a 3-cycle (with
possible multi-edges).

To show the claim, note that $\lie{g}_{\{ i,j \}}$ is of type $A_1 \times
A_1$. Hence the vector
\[
v := (x^-_{\alpha_i})^{\lambda(h_i) + 1} (x^-_{\alpha_j})^{\lambda(h_j) +
1} m_\lambda \in M(\lambda)
\]

\noindent is maximal by $\lie{sl}_2$-theory and the Serre relations.
Moreover, $v$ has weight $s_i s_j \bullet \lambda = s_i s_j(\lambda +
\rho) - \rho$. Now note by the Kostant partition function for $A_1 \times
A_1$ that
\[
\dim M(\lambda)_\mu = \dim M_{\{ i,j \}}(\lambda)_\mu = 1, \qquad \forall
\mu \in \lambda - \Z_+ \Delta_{\{ i,j \}}.
\]

\noindent Hence $s_i s_j \bullet \lambda \notin \wt \vla$, where $\vla =
M(\lambda) / U(\lie{g}) v$. On the other hand, it is clear by inspection
that $\lambda - \Z_+ \alpha_k \subset \wt (M(\lambda) / U(\lie{g}) v)$
for all $k \in I$, so Equation \eqref{Ewtsimple} fails to hold for
$\vla$.
\end{proof}

\subsection{Weyl Character Formula and simple modules}

Note by Theorem \ref{T2} that $\conv_\R \wt L(\lambda)$ $= \conv_\R \wt
M(\lambda, J_\lambda)$ for all $\lambda \in \lie{h}^*$. A stronger result
was Theorem \ref{Twtgvm}, which showed that $\wt L(\lambda) = \wt
M(\lambda, J_\lambda)$ for all $\lambda$. The even stronger assertion --
namely, whether or not $M(\lambda, J_\lambda)$ is simple -- has also
been studied in detail by Wallach \cite{Wa}, Conze-Berline and Duflo
\cite{CD}, and resolved by Jantzen in \cite{Ja}. See \cite[\S 9.12,
9.13]{H3} for more details. The approach in \cite{H3} starts with a
parabolic subgroup of $W$ and then works with suitable highest weights
$\lambda$, while in this paper the approach is reversed, to start with a
highest weight $\lambda$. Thus for completeness, we quickly discuss a
sufficient condition which is slightly different from the one in
\cite{H3}. In particular, the following result yields weight
multiplicities of a large class of simple highest weight modules.

\begin{theorem}[Weyl Character Formula]\label{Twcf}
Suppose the set $S_\lambda := \{ w \in W : w \bullet \lambda \leq \lambda
\}$ equals $W_{J_\lambda}$. Then $L(\lambda)$ is the unique quotient of
$M(\lambda)$ whose set of weights is $W_{J_\lambda}$-invariant, whence
\begin{equation}\label{Ewcf}
\ch L(\lambda) = \ch M(\lambda, J_\lambda) = \frac{\sum_{w \in
W_{J_\lambda}} (-1)^{\ell(w)} e^{w(\lambda + \rho_I)}}{\sum_{w \in W}
(-1)^{\ell(w)} e^{w(\rho_I)}}.
\end{equation}
\end{theorem}

\noindent Note that this result unifies the cases of dominant integral
and antidominant $\lambda$ (where $S_\lambda = W$ and $J_\lambda = I$, or
$S_\lambda = \{ 1 \}$ and $J_\lambda = \emptyset$ respectively). Equation
\eqref{Ewcf} thus generalizes the usual Weyl character formula for
finite-dimensional simple $\lie{g}$-modules (see also the influential
work \cite{Ja}).

\begin{proof}
If $\wt \vla$ is $W_{J_\lambda}$-invariant, then $\wt_{\{ j \}} \vla$ is
$s_j$-invariant for all $j \in J_\lambda$. Thus
$(x^-_{\alpha_j})^{\lambda(h_j) + 1} v_\lambda = 0$ for all $j \in
J_\lambda$ (where $v_\lambda$ spans $\vla_\lambda$), whence $M(\lambda,
J_\lambda) \twoheadrightarrow \vla$. Now let $\mathcal{O}(\lambda)$
denote the block of the BGG Category $\mathcal{O}$ corresponding to
$\lambda$; in other words, $\mathcal{O}(\lambda)$ is the full subcategory
of all finite length $\lie{h}$-semisimple $\lie{g}$-modules, each of
whose Jordan-Holder factors is $L(w \bullet \lambda)$ for some $w \in W$
(in other words, all simple subquotients have the same central character
as $L(\lambda)$).
Recall that the sets $\{ [L(w \bullet \lambda)] : w \in W \}$ and $\{
[M(w \bullet \lambda)] : w \in W \}$ are $\Z$-bases of the Grothendieck
group of the block $\mathcal{O}(\lambda)$, with unipotent (triangular)
change-of-basis matrices with respect to the usual partial order on
$\lie{h}^*$. Thus, $\ch \vla$ is a $\Z$-linear combination of $\ch
M(\mu)$, with $\mu \in S_\lambda$.

Now proceed as in the proof of the Weyl character formula: if $q :=
\prod_{\alpha \in \Phi^+} (e^{\alpha/2} - e^{-\alpha/2})$ is the usual
Weyl denominator, then using that $\dim \vla_\lambda = 1$, we get:
\[
q * \ch \vla = \sum_{w \in W_{J_\lambda}} c_w q * \ch M(w \bullet
\lambda) = \sum_{w \in W_{J_\lambda}} c_w e^{w(\lambda + \rho_I)}, \qquad
c_1 = 1.
\]

\noindent Now the left side is $W_{J_\lambda}$-alternating, whence so is
the right side. This shows that $c_w = (-1)^{\ell(w)}$, and therefore
that $\ch \vla$ is independent of $\vla$ itself. Since $M(\lambda,
J_\lambda) \twoheadrightarrow \vla \twoheadrightarrow L(\lambda)$ all
have $W_{J_\lambda}$-invariant characters, they must all be equal.
Equation \eqref{Ewcf} now follows from the well-known expansion of the
Weyl denominator.
\end{proof}

\section{Extending the Weyl polytope to (pure) highest weight
modules}\label{Sfer}

We now prove Theorems \ref{T2} and \ref{T3}. The first step is to
identify the ``edges'' of the polyhedron $\conv_\R \wt M(\lambda,
J(\vla))$ for simply-regular $\lambda$. We carry this out in greater
generality.

\begin{theorem}\label{Tgvmedges}
Fix $\lambda \in \lie{h^*}$ and $J' \subset J_\lambda$. If $\lambda(h_j)
\neq 0 \ \forall j \in J'$, then $\conv_\R \wt M(\lambda,J')$ is
$W_{J'}$-invariant, and has extremal rays $\{ \lambda - \R_+ \alpha_i \ :
\ i \notin J' \}$ at the vertex $\lambda$.
\end{theorem}

\noindent In particular, the result holds if $\lambda$ is simply-regular.

\begin{proof}
The proof is in steps. The result is trivial for $J' = I$ by standard
results (see e.g.~\cite{H3}), since in this case $\lambda \in P^+$ and
$M(\lambda,I) = L(\lambda)$. Now note by \cite[Proposition 2.4]{KR} that
\[
\conv_\R \wt M(\lambda,J') = \conv_\R \wt_{J'} M(\lambda,J') - \R_+
(\Phi^+ \setminus \Phi_{J'}^+).
\]

\noindent Hence the extremal rays (i.e., unbounded edges) through
$\lambda$ are contained in $\{ \lambda - \R_+ \mu : \mu \in \R_+ (\Phi^+
\setminus \Phi_{J'}^+) \}$. (Note that every extremal ray passes through
a vertex.) The next step is to reduce this set of candidates to $\{
\lambda - \R_+ \mu : \mu \in \Phi^+ \setminus \Phi_{J'}^+ \}$. But this
is clear: if $\mu = \sum_{\alpha \in \Phi^+ \setminus \Phi_{J'}^+}
r_\alpha \alpha$ with $r_\alpha \geq 0$, and $r \in \R_+$, then using
that $J' \neq I$,
\[
\lambda - r \mu = \lambda - \sum_{\alpha \in \Phi^+ \setminus
\Phi_{J'}^+} r r_\alpha \alpha = \frac{1}{|\Phi^+ \setminus \Phi_{J'}^+|}
\sum_{\alpha \in \Phi^+ \setminus \Phi_{J'}^+} \left( \lambda - r
r_\alpha |\Phi^+ \setminus \Phi^+_{J'}| \cdot \alpha \right).
\]

\noindent Now use this principle again: namely, that extremal rays in a
polyhedron are weak $\R$-faces, so no point on such a ray lies in the
convex hull of points not all on the ray.
Thus, we show that the set of extremal rays through $\lambda$ is $\{
\lambda - \R_+ \alpha_i \ : \ i \notin J' \}$. None of these rays
$\lambda - \R_+ \alpha_i$ is in the convex hull of $\{ \lambda - \R_+
\alpha_{i'} \ : \ i' \in I \setminus \{ i \} \}$. Hence it suffices to
show that for all $\mu \in \Phi^+ \setminus (\Delta \cup \Phi_{J'}^+)$
and $r > 0$, the vector $\lambda - r \mu$ is in the convex hull of points
in $\conv_\R \wt M(\lambda,J')$ that are not all in $\lambda - \R_+ \mu$.
Suppose $\mu \in \Phi^+ \setminus \Phi_{J'}^+$ is of the form
\[
\mu = \sum_{j \in J'} c_j \alpha_j + \sum_{s=1}^k d_s \alpha_{i_s},
\]

\noindent where $c_j, 0 < d_s \in \Z_+$ for some $k > 0$, and $i_s \notin
J'$ for all $s$. Recall the assumption on $\lambda$, which implies that
for all $j \in J'$, $s_j(\lambda) = \lambda - n_j \alpha_j$ for some $n_j
> 0$. Finally, to study $\lambda - r \mu$, define the function $f \in
\Fin(\conv_\R \wt M(\lambda,J'), \R_+)$ via:
\[
D := 1 + r \sum_{j \in J'} \frac{c_j}{n_j}, \qquad f(\lambda - k r d_s
\alpha_{i_s}) := \frac{1}{kD}, \qquad f(\lambda - n_{j_0} \alpha_{j_0} -
r \mu) := \frac{r c_{j_0}}{D n_{j_0}}
\]

\noindent for all $1 \leq s \leq k$ and $j_0 \in J'$, and $f$ is zero
otherwise. (If $r \notin \Z_+$, this can be suitably modified to replace
each point in $\supp(f)$ by its two ``neighboring'' points in the
corresponding weight string through $\lambda$, such that the new function
is supported only on $\wt M(\lambda,J')$.)
Note that $\lambda - \R \mu$ does not intersect $\lambda - n_{j_0}
\alpha_{j_0} - \R \mu$. Straightforward computations now show that
$\ell(f) = 1$ and $\disp(f) = \lambda - r \mu$, so $\lambda - r \mu \in
\conv_\R(\supp(f))$. Now if $\mu \neq \alpha_i$ for some $i \notin J'$,
then either some $c_j > 0$ or $k>1$. But then $\supp(f)$ is not contained
in $\lambda - \R_+ \mu$, so it cannot be an extremal ray.
\end{proof}

\subsection{Connections to Fernando's results and convex hulls of pure
modules}

We next discuss connections between our results and the work of Fernando
\cite{Fe}, where he initiated the classification of irreducible
$\lie{h}$-weight $\lie{g}$-modules with finite weight multiplicities.
(This classification was completed by Mathieu in \cite{Ma}; in his
terminology, the simple highest weight modules $L(\lambda) =
L_{\lie{b}}(\lambda)$ are ``parabolically induced''.)

The following result shows that for every highest weight module $\vla$,
the subset $J(\vla)$ is uniquely determined in the spirit of \cite{Fe} as
follows. To state it, we need the following notation from \cite{Fe}.

\begin{defn}
$P \subset \Phi$ is {\em closed} if $\alpha + \beta \in P$ whenever
$\alpha, \beta \in P$ and $\alpha + \beta \in \Phi$.
Next, given a $\lie{g}$-module $M$, define $\lie{g}[M] := \{
X \in \lie{g} \ : \  \C[X] \cdot m \subset U(\lie{g})m \subset
M \mbox{ is finite-dimensional for all } m \in M \}$.
\end{defn}

\begin{prop}\label{Phwfer}
Given $\lambda \in \lie{h}^*$ and $M(\lambda) \twoheadrightarrow \vla$,
$\lie{g}[\vla]$ equals the parabolic subalgebra $\lie{p}_{J(\vla)}$.
Thus, one recovers $J(\vla)$ from $\vla$ via:
\begin{equation}\label{Efer}
J(\vla) \leftrightarrow \Delta_{J(\vla)} = (- \wt \lie{g}[\vla]) \cap
\Delta.
\end{equation}
\end{prop}

\begin{proof}
Apply \cite[Lemma 4.6]{Fe} and the remarks preceding it. In the notation
of \cite{Fe}, $\vla$ lies in the BGG Category $\mathcal{O} \subset
\mathscr{M}(\lie{g},\lie{h}) \subset \overline{\mathscr{M}}(\lie{g},
\lie{h})$. Hence $\vla$ is $\alpha$-finite (i.e., $x_\alpha$ acts locally
finitely on $\vla$) for all roots $\alpha \in \Phi^+$. By \cite[\S 9.3,
9.4]{H3}, $\vla$ is also $\alpha$-finite for all roots $\alpha \in
\Phi^-_{J(\vla)}$, whereas $\vla$ is not $\alpha$-finite for all $\alpha
\in -\Delta_{I \setminus J(\vla)}$ by Theorem \ref{T1}. Hence by
\cite[Lemma 4.6]{Fe}, the set $F(\vla)$ of roots $\alpha \in \Phi$ such
that $\vla$ is $\alpha$-finite is a closed set containing $\Phi^+ \sqcup
\Phi^-_{J(\vla)}$ and disjoint from $-\Delta_{I \setminus J(\vla)}$. Now
by Lemma $3$ in \cite[Chapter VI.1.7]{Bou}, $F(\vla) = \Phi^+ \sqcup
\Phi^-_{J(\vla)}$, whence $\lie{g}[\vla] = \lie{h} \oplus \lie{n}^+
\oplus \lie{n}^-_{J(\vla)} = \lie{p}_{J(\vla)}$. Equation \eqref{Efer} is
now clear; it also follows directly from (the proof of) Theorem \ref{T1}.
\end{proof}

We now show that the convex hull of $\wt \vla$ is a polyhedron for a
large family of modules $\vla$.

\begin{proof}[Proof of Theorem \ref{T2}]
We break up the proof into steps for ease of exposition. We first show
that the convex hull is a polyhedron; next, we compute the extremal rays
if $\lambda$ is simply-regular; finally, we compute the stabilizer in $W$
of the weights and of their hull.\medskip

\noindent \textbf{Step 1.}
The first assertion (except for the stabilizer subgroup being
$W_{J(\vla)}$) follows from Proposition \ref{Tgvm} if $\vla = M(\lambda,
J')$. This implies the result when $|J_\lambda \setminus J(\vla)| \leq
1$, by Theorem \ref{Twtgvm}.

Next suppose that $\vla$ is pure (see Definition \ref{Dpure}). It now
suffices to show that $\conv_\R \wt \vla = \conv_\R \wt M(\lambda,
J(\vla))$. One inclusion follows from Proposition \ref{Pgvm}. Conversely,
to show that $\conv_\R \wt M(\lambda,J(\vla)) \subset \conv_\R \wt \vla$,
observe by Theorem \ref{T1} that $\wt \vla$ is $W_{J(\vla)}$-stable. Now
the vertices of $\conv_\R \wt M(\lambda, J(\vla))$ are
$W_{J(\vla)}(\lambda)$, and
\[
M(\lambda, J(\vla)) \twoheadrightarrow \vla \twoheadrightarrow
U(\lie{g}_{J(\vla)}) v_\lambda \cong L_{J(\vla)}(\lambda) \cong
U(\lie{g}_{J(\vla)}) m_\lambda.
\]

\noindent (Here, $m_\lambda$ and $v_\lambda$ generate $M(\lambda,
J(\vla))$ and $\vla$ respectively.) Thus, $\conv_\R \wt \vla$ also
contains these vertices. Now recall from \cite[Proposition 2.3]{KR} that
$\wt M(\lambda,J(\vla)) = \wt_{J(\vla)} \vla - \Z_+ (\Phi^+ \setminus
\Phi^+_{J(\vla)})$. We \textbf{claim} that for all vertices $\mu \in
W_{J(\vla)}(\lambda)$ and all $\alpha \in \Phi^+ \setminus
\Phi^+_{J(\vla)}$, the set $(\mu - \Z_+ \alpha) \cap \wt \vla$ is
infinite. (This implies in particular that the set of weights along every
extremal ray is infinite.) Now taking the convex hull (twice) shows the
result.

It remains to show the claim. For this, apply \cite[Proposition 4.17]{Fe}
to the pure module $\vla$. Note by purity and Proposition \ref{Phwfer}
that $\vla$ is $\alpha$-finite if $\alpha \in F := \Phi^+ \sqcup
\Phi^-_{J(\vla)}$, and $\alpha$-free if $\alpha \in T := \Phi^- \setminus
\Phi^-_{J(\vla)}$. Following the proof of \cite[Proposition 4.17]{Fe}
yields that $P = F \cup (T \cap (-T)) = F$, whence $\lie{p}_{\vla}^\pm =
\lie{p}_{J(\vla)}^\pm$. Moreover, the result asserts that the nilradical
of $\lie{p}^-_{J(\vla)}$ is torsion-free on all of $\vla$. This implies
that for all $\mu \in \wt \vla$ and $\alpha \in \Phi^+ \setminus
\Phi^+_{J(\vla)}$, the set $(\mu - \Z_+ \alpha) \cap \wt \vla$ is
infinite.\medskip

\noindent \textbf{Step 2.}
Suppose $\lambda$ is simply-regular. Then the first assertion can be
rephrased via Theorem \ref{Tgvmedges} to say that $\conv_\R \wt \vla =
\conv_\R \wt M(\lambda, J(\vla))$. As shown for pure modules,
$W_{J(\vla)}(\lambda) \subset \wt \vla$. It thus suffices to show -- by
the $W_{J(\vla)}$-invariance of both convex hulls in $\lie{h}^*$ -- that
all extremal rays of $\conv_\R \wt M(\lambda, J(\vla))$ at the vertex
$\lambda$ are also contained in $\conv_\R \wt \vla$.
But by Theorem \ref{Tgvmedges}, the extremal rays at $\lambda$ are $\{
\lambda - \R_+ \alpha_i \ : \ i \notin J(\vla) \}$, and these are indeed
contained in $\conv_\R \wt \vla$ (by Theorem \ref{T1}) since $\lambda -
\Z_+ \alpha_i \subset \wt \vla\ \forall i \notin J(\vla)$. This shows
that $\conv_\R \wt \vla = \conv_\R \wt M(\lambda, J(\vla))$, and hence is
a polyhedron, with extremal rays at $\lambda$ as described.\medskip

\noindent \textbf{Step 3.}
Having computed the convex hull, we next show that the stabilizer $W'$ of
$\conv_\R \wt \vla$ equals $W_{J(\vla)}$. By Theorem \ref{T1},
$W_{J(\vla)} \subset W'$. Conversely, if $w' \in W'$, then $w' \lambda$
is a vertex of the convex polyhedron $\conv_\R \wt \vla = \conv_\R \wt
M(\lambda, J(\vla))$. Thus there exists $w \in W_{J(\vla)}$ such that $w'
\lambda = w \lambda$. Now by \cite[Proposition 2.3]{KR}, $w^{-1} w'$
sends the root string $\lambda - \Z_+ \alpha \subset \wt \vla$ to
$\conv_\R \wt M(\lambda, J(\vla))$ for all $\alpha \in \Phi^+ \setminus
\Phi^+_{J(\vla)}$. But then,
\begin{equation}\label{Estab}
w^{-1} w'(\alpha) \in W(\Phi) \setminus (\Phi^- \sqcup \Phi^+_{J(\vla)})
= \Phi^+ \setminus \Phi^+_{J(\vla)}, \qquad \forall \alpha \in \Phi^+
\setminus \Phi^+_{J(\vla)}.
\end{equation}

\noindent Let $w^{-1} w' = s_{i_1} \cdots s_{i_r}$ be a reduced
expression in $W$. If $w' \notin W_{J(\vla)}$, choose the largest $t$
such that $i_t \notin J(\vla)$. Then by Corollary $2$ to Proposition $17$
in \cite[Chapter VI.1.6]{Bou}, $\beta_t := s_{i_r} \cdots
s_{i_{t+1}}(\alpha_{i_t})$ is a positive root such that $w^{-1}
w'(\beta_t) < 0$. By Equation \eqref{Estab}, $\beta_t \in
\Phi^+_{J(\vla)}$. Since $i_u \in J(\vla)$ for $u>t$, hence $\alpha_{i_t}
\in W_{J(\vla)}(\Phi^+_{J(\vla)}) = \Phi_{J(\vla)}$. This implies that
$i_t \in J(\vla)$, which is a contradiction. Thus no such $w' \in W'
\setminus W_{J(\vla)}$ exists, showing that $W' = W_{J(\vla)}$.

Finally, Theorem \ref{T1} implies that $W_{J(\vla)}$ stabilizes $\wt
\vla$. Moreover, if $w \in W$ stabilizes $\wt \vla$, then it also
stabilizes $\conv_\R \wt \vla$, whence $w \in W_{J(\vla)}$ from the above
analysis.
\end{proof}

\begin{remark}
We have thus provided three different proofs for Theorem \ref{T2} in the
case when $\lambda$ is simply-regular and $\vla = L(\lambda)$ or
$M(\lambda,J')$. One method of proof uses convexity theory as in Theorem
\ref{Tgvmedges}; another uses $\lie{sl}_2$-theory as in Theorem
\ref{Twtgvm} together with results from \cite{H3}; and the third uses
Proposition \ref{Phwfer} and results from \cite{Fe}. More precisely, note
by the discussion following \cite[Remark 2.9]{Fe} that all parabolic
Verma modules $M(\lambda,J')$ are pure (see Definition \ref{Dpure}), as
are all simple modules $L(\lambda)$. Now use the arguments for pure
$\vla$ in the proof of Theorem \ref{T2}.
\end{remark}

\subsection{Relating maximizer subsets and (weak) faces}

We now prove Theorem \ref{T3}. It is clear that every maximizer subset of
a polyhedron is a weak $\bba$-face, hence is $(\{ 2 \}, \{ 1, 2
\})$-closed. To show that it must also contain a vertex requires
additional work. Thus, we first extend the main technical tool used in
\cite{KR}, from subfields $\F \subset \R$ to arbitrary additive subgroups
$\bba$:

\begin{prop}\label{Pkr2}
Fix $0 \neq \bba \subset (\R,+)$. Suppose $Y \subset X \subset
\mathbb{Q}^n \subset \R^n$, and $\conv_\R(X)$ is a polyhedron. Then $Y
\subset X$ is a weak $\bba$-face if and only if $Y = F \cap X$, where $F$
is a face of $\conv_\R(X)$.
\end{prop}

\noindent Thus, $Y$ is independent of $\bba$, and weak $\bba$-faces are a
natural extension of the usual notion of a face. Note that \cite[Theorem
4.3]{KR} was stated for $\bba = \F$ (an arbitrary subfield of $\R$), but
assumed more generally that $X \subset \F^n \subset \R^n$. However,
Proposition \ref{Pkr2} is suitable for the setting of $X = \wt \vla$ as
in this paper, because by Lemma \ref{Lfield}, one can replace $X$ by
$\lambda - \wt \vla \subset \Z_+ \Delta \cong \Z_+^n \subset \R^n \cong
\liehr^*$.

\begin{proof}
By \cite[Theorem 4.3]{KR}, if $Y = F \cap X$, then $Y \subset X$ is a
weak $\R$-face, and hence a weak $\bba$-face from the definitions.
Conversely, if $Y$ is a weak $\bba$-face of $X$, then by Lemma
\ref{Lfield} (dividing $a \cdot \Z \subset \bba$ by $a$, for any $0 < a
\in \bba$), $Y \subset X$ is a weak $\Z$-face, hence a weak
$\mathbb{Q}$-face by Lemma \ref{L0}. Again by \cite[Theorem 4.3]{KR}, $Y
= F \cap X$ for some face $F$ of $\conv_\R(X)$, as desired.
\end{proof}

Next, we mention a result pointed out to us by V.~Chari, which in
particular provides a sufficient condition for a weak face to contain a
vertex. When $\lambda \in P^+$ is also simply-regular, the following
result combined with Theorem \ref{Tface} for $L(\lambda)$, as well as the
$W$-invariance of $\wt L(\lambda)$, shows the main results in \cite{KR}
which classify the (positive) weak faces of $\wt L(\lambda)$.

\begin{lemma}\label{L12}
Suppose $0 \neq \lambda \in P^+$ and a nonempty subset $Y \subset
\wtvla{}$ is $(\{ 2 \}, \{ 1, 2 \})$-closed in $\wt \vla = \wt
L(\lambda)$. Then $Y$ contains a vertex $w(\lambda)$ for some $w \in W$.
\end{lemma}

\begin{proof}
(By V.~Chari.) Since $\wtvla{}$ is $W$-stable, (use Lemma \ref{Lfield}
and translate $Y$; now) assume that $ Y \neq \emptyset$ contains some
$\mu \in P^+$. If $\mu = \lambda$, we are done; otherwise, $\lie n^+
L(\lambda)_\mu \neq 0$, so $\mu + \alpha_i \in \wtvla{}$ for some $i \in
I$. But then, so must $s_{\alpha_i}(\mu + \alpha_i) = \mu + \alpha_i -
\langle \mu,\alpha_i \rangle \alpha_i - 2 \alpha_i$, where $\langle \mu,
\alpha_i \rangle \in \Z_+$. Hence $\mu \pm \alpha_i \in \wtvla{}$, and
since $Y$ is $(\{ 2 \}, \{ 1, 2 \})$-closed, $\mu \pm \alpha_i \in Y$.
Now $w(\mu + \alpha_i) \in P^+$ for some $w \in W$. But then $w(Y)$ has a
strictly larger dominant weight: $w(\mu + \alpha_i) \geq \mu + \alpha_i >
\mu$. Repeat this process in $\wtvla{}$; by downward induction on the
height of $\lambda - \mu$, it eventually stops, and stops at $\mu =
\lambda$. Thus, $\lambda \in w(Y)$ for some $w \in W$, whence
$w^{-1}(\lambda) \in Y$.
\end{proof}

Equipped with these technical results, it is now possible to prove
Theorem \ref{T3}.

\begin{proof}[Proof of Theorem \ref{T3}]
Theorem \ref{T2} easily implies that $(1) \Longleftrightarrow (2)$ using
Proposition \ref{Pkr2} and Lemma \ref{L0}.
(One needs to first translate $Y \subset \wt \vla$ to $\lambda - Y
\subset \lambda - \wt \vla$ via Lemma \ref{Lfield}.) That $(3) \implies
(1)$ follows by Theorem \ref{Tface} and $W_{J(\vla)}$-invariance, since
$w(\wt_J \vla) = (\wt \vla)(w(\rho_{I \setminus J}))$.
Conversely, if $\vla = M(\lambda, J')$, then $(1) \implies (3)$ follows
from \cite[Theorem 1]{KR} and the following fact: \textit{given $\lambda
\in \lie{h}^*$, $M(\lambda) \twoheadrightarrow \vla$, integers $k,l>0$
and $J'_r, J''_s \subset I$ for $1 \leq r \leq k$ and $1 \leq s \leq l$,}
\begin{equation}\label{Ecomp2}
\bigcap_{r=1}^k \wt_{J'_r} \vla \cap \bigcap_{s=1}^l \conv_\R
(\wt_{J''_s} \vla) = \wt_{\cap_r J'_r \cap_s J''_s} \vla.
\end{equation}

\noindent Using the above analysis, it follows that $(1) \implies (3)$ if
$\vla$ is pure or $|J_\lambda \setminus J(\vla)| \leq 1$, since it was
shown in the proof of Theorem \ref{T2} that $\conv_\R \wt \vla = \conv_\R
\wt M(\lambda, J(\vla))$ in both of these cases.

It remains to prove that $(1) \implies (4) \implies (3)$ when $\lambda$
is simply-regular and $\vla$ is any highest weight module. Note that (4)
simply says that $Y$ contains a point in $\wt_{J(\vla)} \vla$ and is $(\{
2 \}, \{ 1, 2 \})$-closed in $\wt \vla$. Now $(1) \implies (4)$ follows
from Theorem \ref{Tface}, since any maximizer subset necessarily contains
a vertex (because the polyhedron $\conv_\R \wt \vla$ has a vertex
$\lambda$ by Theorem \ref{T2}), and all vertices are indeed in
$\wt_{J(\vla)} \vla$. Finally, suppose (4) holds for $Y$. Then $Y \cap
\wt_{J(\vla)} \vla$ is $(\{ 2 \}, \{ 1, 2 \})$-closed in $X_1 :=
\wt_{J(\vla)} \vla$ by Lemma \ref{Lfield}. It follows by Lemma
\ref{Lfacts} that
\[
\varpi_{J(\vla)}(Y) \cap \wt L_{J(\vla)}(\pi_{J(\vla)}(\lambda)) =
\varpi_{J(\vla)}(Y \cap \wt_{J(\vla)} \vla) \subset \wt
L_{J(\vla)}(\pi_{J(\vla)}(\lambda))
\]

\noindent is $(\{ 2 \}, \{ 1, 2 \})$-closed. Hence by Lemma \ref{L12}
applied to $\lie{g}_{J(\vla)}$, $\varpi_{J(\vla)}(Y \cap \wt_{J(\vla)}
\vla)$ contains a vertex of the Weyl polytope of
$\pi_{J(\vla)}(\lambda)$. Again via Lemma \ref{Lfacts}, $Y \cap
\wt_{J(\vla)} \vla$ contains a vertex $w \lambda$ for some $w \in
W_{J(\vla)}$. Thus $w^{-1}(Y)$ is $(\{ 2 \}, \{ 1, 2 \})$-closed in $\wt
\vla$ and contains $\lambda$; moreover, $\lambda - \Delta \subset \wt
\vla$ since $\lambda$ is simply-regular. Hence $w^{-1}(Y) = \wt_J \vla$
for some (unique) subset $J \subset I$, by Theorem \ref{Tface}. This
shows (3).
\end{proof}

\begin{remark}
Note that if $\lambda$ is simply-regular, and either $|J_\lambda
\setminus J(\vla)| \leq 1$ or $\vla = M(\lambda,J')$, then we do not need
to assume the condition $Y \cap \wt_{J(\vla)} \vla \neq \emptyset$ in (4)
in Theorem \ref{T3}. Indeed, use Theorem \ref{Twtgvm} and assume $Y
\subset \wt \vla = \wt M(\lambda, J(\vla))$ is $(\{ 2 \}, \{ 1, 2
\})$-closed and nonempty.
By \cite{KR}, suppose $\mu \in \wt_{J(\vla)} M(\lambda,J(\vla))$ and
$\beta \in \Z_+ (\Phi^+ \setminus \Phi^+_{J(\vla)})$ such that $\mu -
\beta \in Y$. Then $(\mu - \beta) + (\mu - \beta) = \mu + (\mu - 2
\beta)$. Hence $\mu, \mu - 2 \beta \in Y$, so $Y \cap \wt_{J(\vla)}
M(\lambda,J(\vla)) \neq \emptyset$.
\end{remark}

\section{Application 2: Largest and smallest modules with specified hull
or stabilizer}\label{Sappl2}

We now discuss an application which is related to Theorem \ref{T1}.
Notice that the set of highest weight modules is naturally equipped with
a partial order under surjection, and it has unique maximal and minimal
elements $M(\lambda)$ and $L(\lambda)$ respectively. We now show that
this ordering can be refined in terms of the stabilizer subgroup of the
weights, or equivalently, their convex hull. For instance if $\lambda +
\rho_I \in P^+$, then by \cite[Proposition 4.3]{H3}, $M(w \cdot \lambda)
\subset M(\lambda)\ \forall w \in W$, and hence
\[
w_\circ \cdot \lambda = w_\circ(\lambda) - 2 \rho_I \quad \implies \quad
\conv_\R \wt M(\lambda) = \conv_\R \wt (M(\lambda) / M(w_\circ \cdot
\lambda)) = \lambda - \R_+ \Delta.
\]

\noindent In fact, there is a unique ``smallest'' highest weight module
whose weights have this same convex hull -- equivalently, whose set of
weights has trivial stabilizer subgroup in $W$. We now prove our last
main result, which generalizes this fact.

\begin{proof}[Proof of Theorem \ref{Tminmax}]
Clearly, $(1) \implies (2) \implies (3)$ by Theorem \ref{T2}. If (3)
holds, define $\mu_j := \lambda - (\lambda(h_j) + 1) \alpha_j$ for all
$j$. Then $\mu_j = s_j(\lambda + \alpha_j)$, so $\mu_j \notin \wt \vla$
for $j \in J'$. Now if $m_\lambda, v_\lambda$ span $M(\lambda)_\lambda$
and $\vla_\lambda$ respectively, then $M(\lambda)_{\mu_j} = \C \cdot
(x_{\alpha_j}^-)^{\lambda(h_j) + 1} m_\lambda$ (using the Kostant
partition function), whence $(x_{\alpha_j}^-)^{\lambda(h_j) + 1}
v_\lambda = 0\ \forall j \in J'$.
Thus $M(\lambda,J') \twoheadrightarrow \vla$, which also implies that
$M_{\max}(\lambda,J') = M(\lambda,J')$.

We show the rest of the implication $(3) \implies (4)$ case-by-case.
First if $J' = J_\lambda$, then $M_{\min}(\lambda,J_\lambda) :=
L(\lambda)$ works by Theorem \ref{T2}. We now show that if $\lambda$ is
simply-regular or $J' = \emptyset$, then there exists
$M_{\min}(\lambda,J')$ as in (4), and moreover, $\conv_\R \wt
M_{\min}(\lambda,J') = \conv_\R \wt M(\lambda,J')$. This would imply that
$(4) \implies (1)$ by the ``intermediate value property'' of convex hulls.

Define $\mathbb{M}(\lambda,J')$ to be the set of all nonzero $M(\lambda)
\twoheadrightarrow \vla$ such that $\conv_\R \wt \vla$ is invariant under
$W_{J'}$ but not a larger parabolic subgroup of $W$. Given $\vla \in
\mathbb{M}(\lambda,J')$, let $K_{\vla}$ denote the kernel of the
surjection $: M(\lambda) \twoheadrightarrow \vla$. Now given such a
$\vla$ and $i \in I$, suppose $(x_{\alpha_i}^-)^n m_\lambda \in K_{\vla}$
for some $n \geq 0$. If $i \in J_\lambda$ and $n \leq \lambda(h_i)$ or $i
\notin J_\lambda$ then $m_\lambda \in K_{\vla}$ by $\lie{sl}_2$-theory,
which is false since $\vla \neq 0$.
Otherwise if $i \in J_\lambda$ and $n > \lambda(h_i)$, then
$(x_{\alpha_i}^-)^{\lambda(h_i) + 1} m_\lambda \in K_{\vla}$ by
$\lie{sl}_2$-theory, whence $M(\lambda, J' \cup \{ i \})
\twoheadrightarrow \vla$. By \cite[\S 9.3, 9.4]{H3}, $\vla \in
\mathbb{M}(\lambda,J')$ is stable under $W_{J' \cup \{ i \}}$, so $i \in
J'$.

We conclude that for all $\vla \in \mathbb{M}(\lambda,J')$,
$(x_{\alpha_i}^-)^n m_\lambda \notin K_{\vla}$ for all $n \geq 0$ and $i
\notin J'$. Since $K_{\vla} \subset M(\lambda)$ is a weight module and
since
\[
\dim (K_{\vla})_{\lambda - n \alpha_i} \leq \dim M(\lambda)_{\lambda - n
\alpha_i} = 1 \qquad \forall i \in I, n \in \Z_+,
\]

\noindent hence $(K_{\vla})_{\lambda - n \alpha_i} = 0\ \forall i \notin
J', n \in \Z_+$. Define $K(\lambda,J') := \sum_{\vla \in
\mathbb{M}(\lambda,J')} K_{\vla}$; then
\begin{equation}\label{Eminmax}
\lambda - \Z_+ \alpha_i \subset \wt M(\lambda,J') / K(\lambda,J') \qquad
\forall \lambda \in \lie{h}^*,\ J' \subset J_\lambda,\ i \notin J'.
\end{equation}

Now suppose (as per the assumptions of the theorem) that $\lambda$ is
simply-regular or $J' = \emptyset$, and $\vla \in
\mathbb{M}(\lambda,J')$. Define
$M_{\min}(\lambda,J') := M(\lambda,J') / K(\lambda,J')$; then
$M_{\max}(\lambda,J'), M_{\min}(\lambda,J') \in \mathbb{M}(\lambda,J')$
by Equation \eqref{Eminmax} and \cite[\S 9.3,9.4]{H3}. Moreover, it is
clear that $\vla \twoheadrightarrow M_{\min}(\lambda,J')$ for all $\vla
\in \mathbb{M}(\lambda,J')$.
Now note by Theorem \ref{T2} that the extremal rays of $\conv_\R
M(\lambda,J')$ at $\lambda$ are $\lambda - \R_+ \Delta_{I \setminus J'}$.
Hence $\conv_\R \wt M_{\max}(\lambda,J') = \conv_\R \wt
M_{\min}(\lambda,J')$. This also shows that $(4) \implies (1)$.
(Moreover, if there is an ``overlap'' in the sense that $\lambda$ is
simply-regular and $J' = J_\lambda$, then $L(\lambda) \in
\mathbb{M}(\lambda,J')$, whence $L(\lambda) \twoheadrightarrow
M_{\min}(\lambda,J') \twoheadrightarrow L(\lambda)$. Thus,
$M_{\min}(\lambda,J')$ is indeed well-defined.)
\end{proof}

\begin{remark}
Note that if $J' \subset J_\lambda$ and $|J_\lambda \setminus J'| \leq
1$, then Theorem \ref{Twtgvm} implies that $\wt : \mathbb{M}(\lambda,J')
\to (\{ 0, 1 \}^{\lambda - \Z_+ \Delta})^{W_{J'}}$ is constant (where
$\mathbb{M}(\lambda,J')$ was defined in the proof of Theorem
\ref{Tminmax}).
\end{remark}

\appendix
\section{Paths between comparable weights in highest weight modules}

In this section, we explain how a well-known result on root systems is a
special case of a phenomenon that occurs in all highest weight modules.
The results in this appendix are not used in the rest of this paper,
though we indicate how they may be used in proving Proposition
\ref{Phwfer}.

Begin by recalling Proposition $19$ in \cite[Chapter VI.1.6]{Bou}: every
positive root $\mu' \in \Phi^+$ can be written as a sum $\mu' =
\alpha_{i_1} + \cdots + \alpha_{i_k}$ of simple roots such that each
partial sum $\alpha_{i_1} + \cdots + \alpha_{i_l}$ is a positive root for
each $1 \leq l \leq k$. We now claim more strongly that it is possible to
rearrange these (possibly repeated) simple roots in such a way that any
of them occurs as $\alpha_{i_1}$. Even more generally, we will show below:

\begin{prop}\label{Prootposet}
Recall the usual partial order on $\lie{h}^*$: $\mu \leq \mu'$ if $\mu' -
\mu \in Q^+ = \Z_+ \Delta$. Suppose $\mu \leq \mu' \in \Phi^+$ (and $\mu$
is a simple root). Then there exists a sequence $i_1, \dots, i_N \in I$
such that for all $0 \leq l \leq N$, $\mu_l := \mu' - \sum_{j=1}^l
\alpha_{i_j} \in \Phi^+$ and $\mu_N = \mu$.
\end{prop}

\noindent Such a result would in fact imply the fact from \cite{Bou} that
was used to show Proposition \ref{Phwfer} -- and by extension, Theorem
\ref{T2} and hence Theorem \ref{T3} as well, for all pure modules $\vla$
(e.g., $L(\lambda)$).

More generally, one can ask the same question for every highest-weight
module $\vla$ for $\lambda \in \lie{h}^*$ (e.g., in every
finite-dimensional simple module):

\begin{qn}\label{Qchain}
Fix $\lambda \in \lie{h}^*$ and $M(\lambda) \twoheadrightarrow \vla$, and
suppose $\mu \leq \mu' \in \wt \vla$. Can we find a sequence of weights
$\mu_j \in \wt \vla$ such that each $\mu_j - \mu_{j+1}$ is a simple root,
and $\mu' = \mu_0 > \mu_1 > \dots > \mu_N = \mu$?
\end{qn}

This question is very general; one can ask it in special cases such as
Verma modules or finite-dimensional modules $\vla = L(\lambda)$ for
$\lambda \in P^+$. (A special case of this is the adjoint representation
as in Proposition \ref{Prootposet}.) Although this result is quite
natural to expect, we are not aware of a reference in the literature
where it is proved. S.~Kumar has communicated a proof to us in the
finite-dimensional case, which we now reproduce, and extend suitably to
all $\vla$ using Theorem \ref{T1}. We thus answer this question
positively in a large number of cases, which include all of the above
examples, as well as others such as \cite[Lemma 4.1]{LCL}.

\begin{theorem}\label{Tkumar}
Suppose $\lambda \in \lie{h}^*$ and $M(\lambda) \twoheadrightarrow \vla$.
Also assume that in the usual partial order on $\lie{h}^*$, $\lambda \geq
\mu' \geq \mu \in \wt \vla$, and that one of the following occurs:
\begin{enumerate}
\item One of these two inequalities is an equality;
\item $\mu' - \mu \in \Delta \cup \Z_+ \Delta_{J(\vla)}$;
\item $|J_\lambda \setminus J(\vla)| \leq 1$ (e.g., $\vla = L(\lambda)$
is simple); or
\item $\vla = M(\lambda, J')$ is a parabolic Verma module, for some $J'
\subset J_\lambda$.
\end{enumerate}

\noindent Then there exists a sequence of weights $\mu_j \in \wt \vla$
such that
\begin{equation}\label{Echain}
\mu' = \mu_0 > \mu_1 > \dots > \mu_N = \mu, \qquad \mu_j - \mu_{j+1} \in
\Delta\ \forall j.
\end{equation}
\end{theorem}

\noindent It is also easy to check that the result holds in other cases:
\begin{itemize}
\item When $\lie{g} = \lie{sl}_2$ and $\vla$ is arbitrary (since every
$\vla$ is a parabolic Verma module).
\item If $\lambda, \mu' \in P^+$ and $\mu \in \conv_\R W (\lambda)$. In
this case, use Theorem \ref{Tklv} to show that $\mu, \mu' \in \wt
L(\lambda) \subset \wt \vla$. Now use Theorem \ref{Tkumar} for
$L(\lambda) = M(\lambda,I)$.
\end{itemize}

\begin{proof}
(The technical heart of this result is the proof assuming the second
condition; this was originally shown by S.~Kumar for finite-dimensional
$\vla$, in the same manner as below.)
Define any pair $(\mu, \mu')$ that satisfies Equation \eqref{Echain} to
be \textit{admissible}. If (1) holds or $\mu' - \mu \in \Delta$, then the
result follows from Lemma \ref{Lweights} (or is obvious). Now suppose (2)
holds and $\mu \leq \mu' \in \wt \vla$ is an inadmissible pair such that
$\mu' - \mu \in \Z_+ \Delta_{J(\vla)}$; we will then arrive at a
contradiction. Choose inadmissible $\mu \leq \mu' \in \wt \vla$ such that
$2 \leq \hgt(\mu' - \mu)$ is minimal. Further refine this choice such
that $1 \leq \hgt(\lambda - \mu')$ is minimal. Now define
\[
J := \supp (\mu' - \mu) = \{ i \in I : (\mu' - \mu, \omega_i) \neq 0 \}.
\]

\noindent Then $J \subset J(\vla)$ by assumption, so by Theorem \ref{T1},
$\wt \vla$ is $w^J_\circ$-stable, where $w^J_\circ \in W_J$ is the
longest element. 
Choose a nonzero vector $v_{\mu'} \in \vla_{\mu'}$. If $\lie{n}_J^-
v_{\mu'} \neq 0$, then there exists $j \in J$ such that $\mu' - \alpha_j
\in \wt \vla$. Since $\mu' - \alpha_j \geq \mu$, the pair $(\mu, \mu' -
\alpha_j)$ is admissible, whence so is $(\mu,\mu')$, a contradiction.
Hence $\lie{n}^-_J v_{\mu'} = 0$, whence $v_{\mu'}$ generates a lowest
weight $\lie{g}_J$-submodule. Since $\vla \in \mathcal{O}$, $U(\lie{g}_J)
v_{\mu'}$ is a finite-dimensional lowest weight $\lie{g}_J$-module,
whence it is in fact simple and isomorphic to $L_J(w^J_\circ \mu')$.
Moreover, $-\mu'(h_j) \in \Z_+$ for all $j \in J$ and hence,
\begin{equation}\label{Ekumar}
\mu' \leq w^J_\circ \mu' < w^J_\circ \mu \leq \lambda, \qquad
\hgt(w^J_\circ \mu - w^J_\circ \mu') = \hgt(\mu' - \mu), \qquad
\hgt(\lambda - w^J_\circ \mu) < \hgt(\lambda - \mu').
\end{equation}

\noindent Hence the pair $(w^J_\circ \mu', w^J_\circ \mu)$ is admissible,
whence there exists a chain of weights as asserted. As mentioned earlier
in this proof, $\wt \vla$ is $w^J_\circ$-stable, so applying $w^J_\circ$
to this chain of weights for $(w^J_\circ \mu', w^J_\circ \mu)$ yields the
desired chain from $\mu'$ to $\mu$ in $\wt \vla$. This is a
contradiction, and we are done if (2) holds.

Next if (3) holds, then $\wt \vla = \wt M(\lambda, J(\vla))$ by Theorem
\ref{Twtgvm}, so the result follows  from (4).
Finally, if $\vla = M(\lambda,J')$ for $J' \subset J_\lambda$, then
suppose $\mu = \mu' - \sum_{i \in I} n_i \alpha_i$ for some choice of
integers $n_i \in \Z_+$. By \cite[Proposition 2.3]{KR}, $\wt
M(\lambda,J')$ is stable under subtracting $\alpha_i$ for $i \notin J'$,
so we obtain a chain in $\wt M(\lambda,J')$ from $\mu'$ to $\mu'' := \mu'
- \sum_{i \notin J'} n_i \alpha_i$. Now apply the above analysis to
\[
\vla = M(\lambda, J'), \qquad \mu' \leadsto \mu'', \qquad \mu \leq \mu''
\in \wt \vla.
\]

\noindent This yields the desired chain in $\wt M(\lambda, J')$ from
$\mu''$ to $\mu$.
\end{proof}

\subsection*{Acknowledgments}
I would like to thank Michel Brion, Daniel Bump, Vyjayanthi Chari, James
Humphreys, and Shrawan Kumar for their stimulating conversations,
suggestions, and references, which were very beneficial at different
stages of this project.
I also thank Tullia Dymarz for pointing me to the reference \cite{Gr}
about distortion, as well as Brian Boe, Gurbir Dhillon, James Lepowsky,
and Volodymyr Mazorchuk for useful discussions.
Finally, I thank the referee for carefully going through the paper, and
for useful comments and suggestions that helped improve the exposition.



\end{document}